\documentclass{amsart}
\usepackage{amssymb}

\usepackage{amsmath}
\usepackage{amscd}

\newtheorem{theorem}{Theorem}
\newtheorem{lemma}{Lemma}
\newtheorem{proposition}{Proposition}
\newtheorem{corollary}{Corollary}
\newtheorem{remark}{Remark}

\input tcilatex

\begin{document}
\title[]{$J$-embeddable reducible surfaces (enlarged version)}
\author{Alberto Alzati}
\address{Dipartimento di Matematica Univ. di Milano\\
via C.\ Saldini 50 20133-Milano (Italy)}
\email{alberto.alzati@unimi.it}
\author{Edoardo Ballico}
\address{Dipartimento di Matematica univ. di Trento\\
via Sommarive 14 38123-Povo (TN) (Italy) }
\email{ballico@science.unitn.it}
\thanks{}
\date{December, 18 2009}
\maketitle

\begin{abstract}
We give a complete classificaton of J-embeddable surfaces, i.e. surfaces
whose secant varieties have dimension at most $4.$
\end{abstract}

\section{Introduction}

Let $\Bbb{P}^{n}$ be the $n$-dimensional complex projective space. In this
paper a variety will be always a non degenerate, reduced subvariety of $\Bbb{%
P}^{n},$ of pure dimension. Surfaces and curves will be subvarieties of
dimension $2$ or $1$, respectively.

In \cite{j} the author introduces the definition of $J$-embedding: for any
subvariety $V\subset $ $\Bbb{P}^{n}$ and for any $\lambda $-dimensional
linear subspace $\Lambda \subset \Bbb{P}^{n}$ we say that $V$ projects
isomorphically to $\Lambda $ if there exists a linear projection $\pi _{%
\mathcal{L}}:\Bbb{P}^{n}--->\Lambda ,$ from a suitable $(n-\lambda -1)$
-dimensional linear space $\mathcal{L}$, disjoint from $V$, such that $\pi _{%
\mathcal{L}}(V)$ is isomorphic to $V.$ We say that $\pi _{\mathcal{L}|V}$ is
a $J$-embedding of $V$ if $\pi _{\mathcal{L}|V}$ is injective and the
differential of $\pi _{\mathcal{L}|V}$ is finite-to one (see \cite{j}, 1.2).

In this paper we want to give a complete classification of $J$-embeddable
surfaces. More precisely we prove (see Lemma \ref{lemVeronese} and
Proposition \ref{primo passo}) the following:

\begin{theorem}
\label{teoint} Let $V$ be a non degenerate, surface in $\Bbb{P}^{n}$, $n\geq %
5.$ Assume that for a generic $4$-dimensional linear subspace $\Lambda
\subset \Bbb{P}^{n}$ the linear projection $\pi _{\mathcal{L}}:\Bbb{P}%
^{n}--->\Lambda $ is such that $\pi _{\mathcal{L}|V}$ is a $J$-embedding of $%
V,$ and that $V$ has at most two irreducible components, then $V$ is in the
following list:

$1)$ $V$ is the Veronese surface in $\Bbb{P}^{5};$

$2)$ $V$ is an irreducible cone;

$3)$ $V$ is the union of a Veronese surface in $\Bbb{P}^{5}$ and a tangent
plane to it;

$4)$ $V$ is the union of two cones having the same vertex;

$5)$ $V$ is the union of a cone with vertex a point $P$ and a plane passing
though $P;$

$6)$ $V$ is the union of :

- an irreducible surface $S,$ such that the dimension of its linear span $%
\left\langle S\right\rangle $ is $4$ and $S$ is contained in a $3$
-dimensional cone having a line $l$ as vertex,

- a plane cutting $\left\langle S\right\rangle $ along $l.$
\end{theorem}

Note that $6)$ is a particular case of Example \ref{e1bis}.

For $J$-embeddable surfaces having at least three irreducible components we
also get a reasonable classification, by distinguishing: the case in which
there exists at least a component having a linear span of dimension at least 
$5,$ see Corollary \ref{cor3span5}; the case in which there exists at least
a component with a $4$-dimensional span, see Theorems \ref{teocon4} and \ref
{teocon4nuovo}; the case in which all components have a span of dimension at
most $3,$ and there exists a pair spanning a linear space of dimension at
least $5,$ see Theorem \ref{teolungo}; the case in which all components have
a span of dimension at most $3$ and every pair has a span of dimension at
most $4,$ see Theorem \ref{teosolo3}.

\section{Notation-Definitions}

If $M\subset \Bbb{P}^{n}$ is any scheme, $M\simeq \Bbb{P}^{k}$ means that $M$
is a $k$-dimensional linear subspace of $\Bbb{P}^{n}.$

$V_{reg}:=$ subset of $V$ consisting of smooth points.

$\left\langle V_{1}\cup ...\cup V_{r}\right\rangle :=$ linear span in $\Bbb{P%
}^{n}$ of the subvarieties $V_{i}\subset \Bbb{P}^{n},$ $i=1,...,r.$

$Sec(V):=\overline{\{\bigcup\limits_{v_{1}\neq v_{2}\in V}\left\langle
v_{1}\cup v_{2}\right\rangle \}}\subset \Bbb{P}^{n}$ for any irreducible
subvariety $V\subset \Bbb{P}^{n}.$

$[V;W]:=\overline{\{\bigcup\limits_{v\in V,w\in W,v\neq w}\left\langle v\cup
w\right\rangle \}}\subset \Bbb{P}^{N}$ for any pair of distinct irreducible
subvarieties $V,W\subset \Bbb{P}^{n}.$

In case $V=W,$ $[V;V]=Sec(V).$ In case $V=W$ is a unique point $P$ we put $%
[V;W]=P$.

In case $V$ is reducible, $V=V_{1}\cup ...\cup V_{r},$ $Sec(V)=\{\bigcup%
\limits_{i=1}^{r}\bigcup\limits_{j=1}^{r}[V_{i};V_{j}]\}.$

In case $V$ and $W$ are reducible, without common components, $V=V_{1}\cup
...\cup V_{r},$ $W=W_{1}\cup ...\cup W_{s},$ we put $[V;W]:=\bigcup%
\limits_{i=1}^{r}\bigcup\limits_{j=1}^{s}[V_{i};W_{j}]$ (with the reduced
scheme structure).

$T_{P}(V):$ $=$ embedded tangent space at a smooth point $P$ of $V.$

$\mathcal{T}_{v}(V):$ $=$ tangent star to $V$ at $v:$ it is the union of all
lines $l$ in $\Bbb{P}^{n}$ passing through $v$ such that there exist at
least a line $\left\langle v^{\prime }\cup v^{\prime \prime }\right\rangle
\rightarrow l$ when $v^{\prime },v^{\prime \prime }\rightarrow v$ with $%
v^{\prime },v^{\prime \prime }\in V.$ (see \cite{j} page. 54).

$Vert(V):=\{P\in V|$ $[P;V]=V\}.$

Let us recall that $Vert(V)$ is always a linear space, moreover\newline
$Vert(V)=\bigcap\limits_{P\in V}(T_{P}(V))$, (see \cite{a2}, page. 17).

We say that $V$ is a cone of vertex $Vert(V)$ if and only if $V$ is not a
linear space and $Vert(V)\neq \emptyset .$ If $V$ is a cone the codimension
in $V$ of $Vert(V)$ is at least two.

\begin{remark}
\label{remconi} If $V$ is an irreducible surface, not a plane, for which
there exists a linear space $L,$ such that for any generic point $P\in V,$ $%
T_{P}(V)\supseteq L,$ then $L$ is a point and $V$ is a cone over an
irreducible curve with vertex $L$ (see \cite{a2}, page. 17).
\end{remark}

Caution: in this paper we distinguish among two dimensional cones and
planes, so that a two dimensional cone will have a well determined point as
vertex.

For any subvariety $V\subset \Bbb{P}^{n}$ let us denote by

\begin{center}
$V^{*}:=\overline{\{H\in \Bbb{P}^{n*}|\text{ }H\supseteq T_{P}(V)\text{ for
some point }P\in V_{reg}\}}$
\end{center}

the dual variety of $V,$ where $\Bbb{P}^{n*}$ is the dual projective space
of $\Bbb{P}^{n}$ and $H$ is a generic hyperplane of $\Bbb{P}^{n}$. Let us
recall that $(V^{*})^{*}=V$.

\section{Background material}

In this section we collect some easy remark about the previous definitions
and some known results which will be useful in the sequel.

\begin{proposition}
\label{Jem} Let $V$ be any subvariety of $\Bbb{P}^{n}$ and let $P$ be a
generic point of $\Bbb{P}^{n}.$ If $P\notin [V;V]$ then $\pi _{P|V}$ is a $J$%
-embedding of $V$.
\end{proposition}

\begin{proof}
See Proposition 1.5 c) of \cite{z}, chapter II, page 37.
\end{proof}

\begin{corollary}
\label{corjoin} Let $V$ be any surface of $\Bbb{P}^{n},$ $n\geq 5,$ and let $%
\Lambda $ be a generic $4$-dimensional linear space of $\Bbb{P}^{n}.$ There
exists a $J$-embedding $\pi _{P|V}$ for $V,$ from a suitable $(n-5)$%
-dimensional linear space of $\Bbb{P}^{n}$ into $\Lambda \simeq \Bbb{P}^{4}$
, if and only if\newline
$\dim [Sec(V)]\leq 4.$
\end{corollary}

\begin{proof}
Apply Proposition \ref{Jem}. See also Theorem 1.13 c) of \cite{z}, chapter
II, page 40.
\end{proof}

\begin{corollary}
\label{corovvio} Let $V=V_{1}\cup ...\cup V_{r}$ be a reducible surface in $%
\Bbb{P}^{n}$, $n\geq 5,$ and let $\Lambda $ be a generic $4$-dimensional
linear space of $\Bbb{P}^{n}.$ There exists a $J$-embedding $\pi _{P|V}$ for 
$V,$ from a suitable $(n-5)$-dimensional linear space of $\Bbb{P}^{n}$ into $%
\Lambda \simeq \Bbb{P}^{4}$, if and only if $\dim ([V_{i};V_{j}])\leq 4$ for
all $i,j=1,...,r,$ including cases $i=j.$
\end{corollary}

\begin{proof}
Look at the definition of $Sec(V)$ and apply Corollary \ref{corjoin}.
\end{proof}

\begin{lemma}
\label{lemJ} For any pair of distinct irreducible subvarieties $V,W\subset 
\Bbb{P}^{n}:$

$1)$ if $V$ and $W$ are linear spaces $[V;W]=$ $\left\langle
V,W\right\rangle ;$

$2)$ if $V$ is a linear space$,$ $[V;W]$ is a cone, having $V$ as vertex;

$3)$ $\left\langle [V;W]\right\rangle $ $=\left\langle \left\langle
V\right\rangle \cup \left\langle W\right\rangle \right\rangle ;$

$4)$ $\left\langle V\right\rangle =\left\langle \bigcup\limits_{P\in
V}T_{P}(V)\right\rangle ,$ $P$ generic point of $V;$

$5)$ $[V;[W;U]]=[[V;W];U]=\overline{\{\bigcup\limits_{v\in V,w\in W,u\in
U,v\neq w,v\neq u,u\neq w}\left\langle v\cup w\cup u\right\rangle \}}$, for
any other irreducible subvariety $U$ distinct from $V$ and $W.$
\end{lemma}

\begin{proof}
Immediate consequences of the definitions of $[V;W]$ and $\left\langle
V\right\rangle .$
\end{proof}

Let us recall the Terracini's lemma:

\begin{lemma}
\label{lemT} Let us consider a pair of irreducible subvarieties $V,W\subset 
\Bbb{P}^{n}$ and a generic point $R\in [V;W]$ such that $R\in \left\langle
P\cup Q\right\rangle ,$ with $P\in V$ and $Q\in W;$ then $T_{R}([V;W])=$ $%
\left\langle T_{P}(V)\cup T_{Q}(W)\right\rangle $ and $\dim ([V;W])=\dim
(\left\langle T_{P}(V)\cup T_{Q}(W)\right\rangle ).$
\end{lemma}

\begin{proof}
See Corollary 1.11 of \cite{a1}.
\end{proof}

The following lemmas consider the join of two irreducible varieties of low
dimensions.

\begin{lemma}
\label{lemcurve} Let $C,$ $C^{\prime }$ be irreducible distinct curves in $%
\Bbb{P}^{n},$ $n\geq 2$, then $\dim ([C;C^{\prime }])=3$ unless $C$ and $%
C^{\prime }$ are plane curves, lying on the same plane, in this case\newline
$\dim ([C;C^{\prime }])=2$.
\end{lemma}

\begin{proof}
The claim follows from Corollary 1.5 of \cite{a1} with $r=2.$
\end{proof}

\begin{lemma}
\label{lemcursup} Let $C$ be an irreducible curve, not a line, and let $B$
be an irreducible surface in $\Bbb{P}^{n},$ $n\geq 2$. Then:

$i)$ $\dim ([C;B])\leq 4;$

$ii)$ $\dim ([C;B])=3$ if and only if $\left\langle C\cup B\right\rangle
\simeq \Bbb{P}^{3};$

$iii)$ $\dim ([C;B])=2$ if and only if $B$ is a plane and $C\subset B.$
\end{lemma}

\begin{proof}
$i)$ Obvious.

$ii)$ If $\dim ([C;B])=3=1+\dim (B),$ by Proposition 1.3 of \cite{a1}, we
have $C\subseteq Vert([C;B]).$ If $[C;B]\simeq \Bbb{P}^{3}$ then $%
\left\langle C\cup B\right\rangle \simeq \Bbb{P}^{3}$ and we are done. If
not the codimension of $Vert([C;B])$ in $[C;B]$ is at least $2$ (see \cite
{a1} page. 214). Hence $\dim \{Vert([C;B])\}\leq 1$. Hence $Vert([C;B])=C,$
but this is a contradiction as $C$ is not a line and $Vert([C;B])$ is a
linear space.

$iii)$ If $\dim ([C;B])=2=1+\dim (C)$ then Proposition 1.3 of \cite{a1}
implies $B\subseteq Vert([C;B]).$ In this case $Vert([C;B])=$ $[C;B]=B$.
Hence $B$ is a plane and necessarily $C\subset B$ by Lemma \ref{lemcurve}.
\end{proof}

\begin{lemma}
\label{lemretta} Let $B$ be an irreducible surface, and $l$ any line in $%
\Bbb{P}^{n},$ $n\geq 2$. Then:

$i)$ $\dim ([l;B])\leq 4;$

$ii)$ $\dim ([l;B])=3$ if and only if $\left\langle l\cup B\right\rangle
\simeq \Bbb{P}^{3}$ or $B$ is contained in a cone $\Xi $ having $l$ as
vertex and an irreducible curve $C$ as a basis.

$iii)$ $\dim ([l;B])=2$ if and only if $B$ is a plane and $l\subset B.$
\end{lemma}

\begin{proof}
$i)$ Obvious.

$ii)$ If $\dim ([l;B])=3=1+\dim (B),$ by Proposition 1.3 of \cite{a1}, we
have $l\subset Vert([l;B]).$ If $[l;B]\simeq \Bbb{P}^{3}$ we have $%
\left\langle l\cup B\right\rangle \simeq \Bbb{P}^{3}$, if not the
codimension of $Vert([l;B])$ in $[l;B]$ is at least $2$ (see \cite{a1} page.
214). Hence $\dim \{Vert([l;B])\}\leq 1$, hence $Vert([l;B])=l$ and $\Xi $
is exactly $[l;B].$ Note that $\dim ([l;B])=3$ if and only if $l\cap
T_{P}(B)\neq \emptyset $ for any generic point $P\in B.$

$iii)$ If $\dim ([l;B])=2=1+\dim (l),$ by Proposition 1.3 of \cite{a1}, we
have $B\subset Vert([l;B]).$ We can argue as in the proof of Lemma \ref
{lemcursup}, $iii).$
\end{proof}

The following Lemmas consider the possible dimensions for the join of two
surfaces according to the dimension of the intersection of their linear
spans. Firstly we consider the case in which one of the two surface is a
plane.

\begin{lemma}
\label{lemJ1} Let $A$ be an irreducible, non degenerate surface in $\Bbb{P}%
^{n},$ $n\geq 3,$ and let $B$ be any fixed plane in $\Bbb{P}^{n}.$ Let $%
A^{\prime }$ be the tangent plane at a generic point of $A_{reg}.$ Then:

$i)$ $\dim ([A;B])=5$ if and only if $A^{\prime }\cap B=\emptyset ;$

$ii)$ $\dim ([A;B])=4$ if and only if $\dim (A^{\prime }\cap B)=0;$

$iii)$ $\dim ([A;B])=3$ if and only if $\dim (A^{\prime }\cap B)=1;$

$iv)$ $\dim ([A;B])=3$ if and only if $\left\langle A,B\right\rangle \simeq 
\Bbb{P}^{3}.$
\end{lemma}

\begin{proof}
As $n\geq 3$, $\dim ([A;B])\geq 3$ and $i)$, $ii)$ and $iii)$ are
consequences of lemma \ref{lemT}. If $\left\langle A,B\right\rangle \simeq 
\Bbb{P}^{3}$ obviously $\dim (A^{\prime }\cap B)=1.$ On the other hand, let
us assume that $\dim (A^{\prime }\cap B)=1$ and let us consider two
different generic points $P,Q\in A\backslash B;$ we have $[A;B]\supseteq
[P;B]\cup [Q;B]$ and $[P;B]\simeq [Q;B]\simeq \Bbb{P}^{3}.$ If $P\notin $ $%
[Q;B]$ we have $\dim ([A;B])\geq 4,$ because $[A;B]$ is irreducible and it
cannot contain the union of two distinct copies of $\Bbb{P}^{3},$
intersecting along a plane, unless $\dim ([A;B])\geq 4,$ but this is a
contradiction with $\dim (A^{\prime }\cap B)=1$ by $ii)$. Hence $P\in
[Q;B]\simeq \Bbb{P}^{3}$ and $A\subseteq [Q;B]\simeq \Bbb{P}^{3}$ as $P$ is
a generic point of $A.$
\end{proof}

\begin{lemma}
\label{lemJ2} Let $A,B$ be two irreducible, surfaces in $\Bbb{P}^{n},$ $%
n\geq 5.$ Let us assume that neither $A$ nor $B$ is a plane. Set $%
L:=\left\langle A\right\rangle \cap \left\langle B\right\rangle $, $%
M:=\left\langle A\cup B\right\rangle ,$ $m:=\dim (M).$ Then:

$i)$ if $L=\emptyset ,$ $\dim ([A;B])=5;$

$ii)$ if $L$ is a point $P,$ $\dim ([A;B])\leq 4$ if and only if $A$ and $B$
are cones with vertex $P;$

$iii)$ if $\dim (L)=1$, $\dim ([A;B])\leq 4$ if and only if:

- there exists a point $P\in L$ such that $A$ and $B$ are cones with vertex $%
P,$ or

- $m\leq 4;$

$iv)$ if $\dim (L)=2,$ $\dim ([A;B])\leq 4$ if and only if:

- there exists a point $P\in L$ such that $A$ and $B$ are cones with vertex $%
P,$ or

- $\dim (\left\langle A\right\rangle )=\dim (\left\langle B\right\rangle )=3$
and $m=4.$
\end{lemma}

\begin{proof}
$i)$ let $A^{\prime }$ be the tangent plane at a generic point of $A_{reg}.$
Let $B^{\prime }$ be the tangent plane at a generic point of $B_{reg}.$ We
have $A^{\prime }\cap B^{\prime }=\emptyset $ so that $i)$ follows from
Lemma \ref{lemT}.

$ii)$ Obviously, in any case, if $A$ and $B$ are cones with a common vertex $%
P$, $A^{\prime }$ and $B^{\prime }$ contain $P$ so that $\dim ([A;B])\leq 4$
by Lemma \ref{lemT}. On the other hand, if $L=P,$ $A^{\prime }\cap B^{\prime
}\neq \emptyset $ only if $A^{\prime }\cap B^{\prime }=P$ and this implies
that the tangent planes at the generic points of $A$ and $B$ contain $P$.
Hence $A$ and $B$ are cones with common vertex $P.$

$iii)$ If $m\leq 4$ obviously $\dim ([A;B])\leq 4.$ Let us assume that $%
m\geq 5$ and $\dim ([A;B])\leq 4.$ Lemma \ref{lemT} implies $A^{\prime }\cap
B^{\prime }\neq \emptyset ,$ while, obviously, $A^{\prime }\cap B^{\prime
}\subseteq L.$ Neither $A^{\prime }$ nor $B^{\prime }$ can contain $L$
because $A$ and $B$ are not planes. Hence $A^{\prime }\cap B^{\prime }$ is a
point $P\in L$ and we can argue as in $ii).$

$iv)$ Let us assume that $\dim ([A;B])\leq 4$ and that $A$ and $B$ are not
cones with a common vertex $P.$ By Lemma \ref{lemT} we have $A^{\prime }\cap
B^{\prime }\neq \emptyset ,$ and, obviously, $A^{\prime }\cap B^{\prime
}\subseteq L.$ As $A$ and $B$ are not cones with a common vertex it is not
possible that $A^{\prime }\cap B^{\prime }$ is a fixed point and it is not
possible that $A^{\prime }\cap B^{\prime }$ is a fixed line because $A$ and $%
B$ are not planes. Hence $\dim (A^{\prime }\cap L)=\dim (B^{\prime }\cap
L)=1 $ and in this case $\dim ([A;L])=\dim ([B;L]=$ $3$ by Lemma \ref{lemJ1} 
$iii).$ It follows that $\dim (\left\langle A\right\rangle )=\dim
(\left\langle B\right\rangle )=3$ by Lemma \ref{lemJ1} $iv)$, hence $m=4.$
\end{proof}

\begin{lemma}
\label{lemJ3} Let $A,B$ be two irreducible surfaces in $\Bbb{P}^{n},$ $n\geq
5.$ Set $L:=\left\langle A\right\rangle \cap \left\langle B\right\rangle $, $%
M:=\left\langle A\cup B\right\rangle ,$ $m:=\dim (M).$ Let us assume that $%
\dim (\left\langle A\right\rangle )=\dim (\left\langle B\right\rangle )=4,$ $%
\dim (L)=3,$ $m=5,$ $\dim ([A;B])\leq 4.$ Then $A$ and $B$ are cones with
the same vertex.
\end{lemma}

\begin{proof}
By Lemma \ref{lemT} we know that for any pair of points $(P,Q)\in
A_{reg}\times B_{reg},$ $\emptyset \neq T_{P}(A)\cap T_{Q}(B)\subseteq L.$
As $(P,Q) $ are generic, we can assume that $P\in A\backslash (A\cap L)$ and 
$Q\in B\backslash (B\cap L),$ so that $l_{P}:=$ $T_{P}(A)\cap L$ and $%
l_{Q}:=T_{Q}(B)\cap L$ are lines, intersecting somewhere in $L.$

$(a)$ Let us assume that $l_{P}\cap l_{P^{\prime }}=\emptyset $ for any
generic pair of points $(P,P^{\prime })\in A\backslash (A\cap L).$ Then the
lines $\{l_{P}|P\in A\backslash (A\cap L),P\in A_{reg}\}$ give rise to a
smooth quadric $\mathcal{Q}$ in $L\simeq \Bbb{P}^{3}$ in such a way that the
lines $\{l_{P}\}$ all belong to one of the two rulings of $\mathcal{Q}.$
Note that $\mathcal{Q}\neq A,$ because they have different spans. Now, for
any smooth point $P\in $ $A\backslash (A\cap L),$ let us consider a generic
tangent hyperplane $H_{P}\subset M$ at $P.$ Obviously $H_{P}\supset T_{P}(A)$
and, as $H_{P}$ is generic, it cuts $L$ only along a plane and this plane
contains $l_{P}.$ Hence it is a tangent plane for $\mathcal{Q}.$ It follows
that $H_{P}$ is also a tangent hyperplane for $\mathcal{Q}$ in $M.$
Therefore $A^{*}\subseteq \mathcal{Q}^{*}$ in $M^{*}.$ If $A$ is not a
developable, ruled surface we have $A^{*}=\mathcal{Q}^{*}$ by looking at the
dimension. Hence $A$ $=(A^{*})^{*}=(\mathcal{Q}^{*})^{*}=\mathcal{\ Q}:$
contradiction.

Now let us assume that $A$ is a developable, ruled surface and let us
consider the curve $C:=A\cap L$, which is a hyperplane section of $A.$ We
claim that the support of $C$ is not a line. In fact $C$ must contain a
directrix for $A$ because $C$ is a hyperplane section of $A$. So that if the
support of $C$ is a line $l$ this line must be a directrix for $A$. Hence a
direct local calculation shows that $l$ is contained in every tangent plane
at points of $A_{reg}.$ It follows that $l_{P}=l$ for any point $P\in
A_{reg}:$ contradiction, and the claim is proved. On the other hand, for a
fixed line $\overline{l_{Q}}$ we can consider $[\overline{l_{Q}};C].$ Since
the support of $C$ is not a line $[\overline{l_{Q}};C]=L,$ moreover $[%
\overline{l_{Q}};C]\subsetneq [\overline{l_{Q}};A]$. Hence $\dim ([\overline{%
l_{Q}};A])\geq 4$, but this is a contradiction with Lemma \ref{lemT} because 
$\overline{l_{Q}}$ $\cap T_{P}(A)\neq \emptyset ,$ for any point $P\in
A_{reg}.$

$(b)$ From $(a)$ it follows that $l_{P}\cap l_{P^{\prime }}\neq \emptyset $
for any generic pair of points $(P,P^{\prime })\in A\backslash (A\cap L).$
It is known (and a very easy exercise) that this is possible only if all
lines $\{l_{P}\}$ pass through a fixed point $V_{A}\in L$ or all lines $%
\{l_{P}\}$ lie on a fixed plane $U_{A}\subset L.$ In the same way we get $%
l_{Q}\cap l_{Q^{\prime }}\neq \emptyset $ for any generic pair of points $%
(Q,Q^{\prime })\in B\backslash (B\cap L)$ and that all lines $\{l_{Q}\}$
pass through a fixed point $V_{B}\in L$ or all lines $\{l_{Q}\}$ lie on a
fixed plane $U_{B}\subset L.$

As for any pairs of points $(P,Q)\in A_{reg}\times B_{reg},$ $\emptyset \neq
T_{P}(A)\cap T_{Q}(B)\subseteq L$, we have only four possibilities:

$1)$ $V_{A}=V_{B}$. Hence $A$ and $B$ are cones having the same vertex
(recall that $T_{P}(A)\supset l_{P}\supset V_{A}$ and $T_{Q}(B)\supset
l_{Q}\supset V_{B})$ and we are done;

$2)$ $V_{A}\in U_{B},$ and all lines $\{l_{Q}\}\subset U_{B}$ pass
necessarily through $V_{A},$ so that $A$ and $B$ are cones having the same
vertex in this case too;

$3)$ $V_{B}\in U_{A}$ and we can argue as in case $2);$

$4)$ there exist two planes $U_{A}$ and $U_{B}.$

If $U_{A}\cap U_{B}$ is a line $l,$ then the generic tangent planes $%
T_{P}(A) $ and $T_{Q}(B)$ would contain $l$ and both $A$ and $B$ would be
planes: contradiction. If $U_{A}=U_{B}$, by Lemma \ref{lemT} we get $\dim
([U_{A};A])=\dim ([U_{B};B])=3$ and they are (irreducible) cones as $U_{A}$
and $U_{B}$ are linear spaces. Hence they are $3$-dimensional linear spaces
containing $A$ and $B,$ respectively: contradiction.
\end{proof}

\section{Examples of J-embeddable surfaces}

In Section 4 we give some examples of $J$-embeddable surfaces and we prove a
result concerning the Veronese surface which will be useful for the
classification.

\begin{example}
\label{e1} Let $W$ be a fixed $2$-dimensional linear subspace in $\Bbb{P}%
^{n} $, $n\ge 5$. Let $m$ be a positive integer such that $1\leq m\leq n-2.$
Let us consider $m$ distinct $3$-dimensional linear subspaces $M_{i}\subset $
$\Bbb{P}^{n},$ $1\le i\le m$, such that $W\subset M_{i}$ for $i=1,...,m$ and 
$\left\langle M_{1}\cup ...\cup M_{m}\right\rangle $ $=\Bbb{P}^{n}$. For
each $i=1,\dots ,m$ fix a reduced surface $D_{i}$ of $M_{i}$ in such a way
that $X:=\cup _{i=1}^{m}D_{i}$ spans $\Bbb{P}^{n}$. We claim that $X$ can be 
$J$-projected into a suitable $\Bbb{P}^{4}.$ By Corollary \ref{corjoin} it
suffices to show that $\dim [Sec(X)]\le 4$. Indeed, $\dim [Sec(D_{i})]\le 3$
for all $i$, while $\dim ([D_{i};D_{j}])\le 4$ for all $i\ne j$, because
every $D_{i}\cup D_{j}$ is contained in the $4$-dimensional linear space $%
\langle M_{i}\cup M_{j}\rangle $.
\end{example}

\begin{example}
\label{e1bis} Let $N$ be a fixed $4$-dimensional linear subspace in $\Bbb{P}%
^{n}$, $n\ge 5$. Let $A_{i}\subset N$ be irreducible surfaces, $i=1,...,s$.
Assume that every $A_{i}$ is contained in the intersection of some $3$%
-dimensional cones $E_{j}\subset N$ having a line $l_{j}$ as vertex and let $%
\{B_{jk_{j}}\}$ be a set of pairwise intersecting planes in $\Bbb{P}^{n}$
such that $B_{jk_{j}}\cap N=l_{j}$, with $j,k_{j}\geq 1.$ Set $%
X:=\{A_{i}\cup B_{jk_{j}}\}.$ We claim that $X$ can be $J$-projected into a
suitable $\Bbb{P}^{4}.$

In fact, by Corollary \ref{corjoin}, it suffices to show that $\dim
[Sec(X)]\le 4$ and the only non trivial check is that $\dim
([A_{i};B_{jk_{j}}])\le 4$ for any $A_{i}$ and for any plane $B_{jk_{j}},$
but this follows from Lemma \ref{lemT} because for any $j$ and for any point 
$P\in (A_{i})_{reg}\cap (E_{j})_{reg}$ the tangent plane $T_{P}(A_{i})$ is
contained in $T_{P}(E_{j})\simeq \Bbb{P}^{3},$ hence $T_{P}(A_{i})\cap
l_{j}\neq \emptyset .$
\end{example}

\begin{example}
\label{e2} Let $Y\subset \Bbb{P}^{5}$ be a Veronese surface. Fix a point $%
P\in Y$ and set $X:=Y\cup T_{P}(Y)$. Let us recall that $\dim [Sec(Y)]=4$.
Hence, by Terracini's lemma, we know that $T_{P}(Y)\cap T_{Q}(Y)\neq
\emptyset $ for any pair of points $P,Q\in Y.$ Therefore $\dim
[Y,T_{p}(Y)]=4 $ and $\dim [Sec(X)]=4$ too. Then we can apply Corollary \ref
{corjoin}.
\end{example}

The following proposition shows that the above example is in fact the only
possibility for a surface $X=Y\cup B$ to have $\dim [Sec(X)]=4,$ where $B$
is any irreducible surface.

\begin{proposition}
\label{propYB} Let $Y\subset \Bbb{P}^{n}$ be a Veronese surface embedded in $%
\langle Y\rangle \simeq \Bbb{P}^{5}$, $n\geq 5,$ and let $B\subset \Bbb{P}
^{n}$ be any irreducible surface. Set $X:=Y\cup B$. Thus $\dim [Sec(X)]=4$
if and only if $B$ is a plane in $\langle Y\rangle ,$ tangent to $Y$ at some
point $P.$
\end{proposition}

\begin{proof}
For the proof it is useful to choose a plane $\Pi $ such that $\langle
Y\rangle \simeq \Bbb{P}^{5}$ is the linear space parametrizing conics of $%
\Pi ,$ i.e. $\left\langle Y\right\rangle \simeq \Bbb{P}[H^{0}(\Pi ,\mathcal{O%
}_{\Pi }(2))].$ Then $Y$ can be considered as the subvariety of $\langle
Y\rangle $ parametrizing double lines of $\Pi $, moreover $Y$ can be also
considered as the $2$-Veronese embedding of $\Pi ^{*}$ via a map $\nu .$

Firstly, let us consider the case in which $B$ is a plane in $\langle
Y\rangle .$ Obviously\newline
$\dim [Sec(X)]=4$ if and only if $\dim [Y;B]=4.$ Note that $\dim [Y;B]=5$ if 
$B\cap Y=\emptyset ,$ because every point $P\in $ $\Bbb{P}^{5}$ is contained
in at least a line intersecting both $B$ and $Y.$ Then we have to consider
all other possibilities for $B\cap Y.$

Let us remark that $\dim [Y;B]=4,$ if and only if the linear projection $\pi
_{B}:\Bbb{P}^{5}--->\Lambda $ is such that $\dim [\overline{\pi
_{B}(Y\backslash B)}]=1,$ where $\Lambda \simeq \Bbb{P}^{2}$ is a generic
plane, disjoint from $B$. In fact $\dim [Sec(X)]=4,$ if and only if $\dim
([B;Y])=4,$ if and only if $\dim (\overline{\bigcup\limits_{y\in Y\backslash
B}\left\langle B\cup y\right\rangle })=4,$ if and only if $\dim [(\overline{%
\bigcup\limits_{y\in Y\backslash B}\left\langle B\cup y\right\rangle })\cap
\Lambda ]=1.$ But $(\overline{\bigcup\limits_{y\in Y\backslash
B}\left\langle B\cup y\right\rangle })\cap \Lambda =$ $\overline{\pi
_{B}(Y\backslash B)}.$

Let us assume that $\dim (B\cap Y)=1.$ It is well known that $Y$ does not
contain lines or other plane curves different from smooth conics. If the
scheme $B\cap Y$ contains a smooth conic $\gamma $, it is easy to see that
the generic fibres of any linear projection as $\pi _{B}$ are $0$
-dimensional. Indeed, by considering the identification $\left\langle
Y\right\rangle \simeq \Bbb{P}[H^{0}(\Pi ,\mathcal{O}_{\Pi }(2))],$ for any
point $P\in Y,$ $T_{P}(Y)$ parametrizes the reducible conics of $\Pi $ whose
components are: a fixed line $r$ of $\Pi $ (such that $P\leftrightarrow
r^{2} $) and any line of $\Pi .$ While $B$ parametrizes the reducible conics
of $\Pi $ having a singular point $Q\in \Pi $ such that the dual line $l\in
\Pi ^{*}$ corresponding to $Q$ is such that $\nu (l)=\gamma .$ Therefore,
for generic $P\in Y,$ $T_{P}(Y)\cap B=\emptyset .$ It follows that $\dim [%
\overline{\pi _{B}(Y\backslash B)}]=2$ and $\dim [Y,B]=5.$ This fact can
also be checked by a direct computation with a computer algebra system, for
instance Macaulay, taking into account that $Y$ is a homogeneous variety, so
that the computation can be made by using a particular smooth conic of $Y.$

Let us assume that $\dim (B\cap Y)=0$ and that $B\cap Y$ is supported at a
point $P\in Y.$ We have to consider three cases:

$i)$ $B$ does not contain any line $l\in T_{P}(Y);$ in this case the
intersection is transversal at $P$ and the projection of $Y$ from $P$ into a
generic $\Bbb{P}^{4}$ gives rise to a smooth cubic surface $Y_{P}$, (recall
that $Y$ has no trisecant lines). The projection of $Y_{P}$ from a line to a
generic plane has generic $0$-dimensional fibres. Hence $\dim [\overline{\pi
_{B}(Y\backslash B)}]=2$ for any generic projection $\pi _{B}$ as above and $%
\dim [Y;B]=5.$

$ii)$ $B$ contains only a line $l\in T_{P}(Y);$ in this case the generic
fibres of any linear projection as $\pi _{B}$ are $0$-dimensional. This fact
can be proved by a direct computation with a computer algebra, for instance
Macaulay; as above the computation can be made by using a particular line of 
$Y$. Hence $\dim $ $\overline{[\pi _{B}(Y\backslash B)}]=2$ and $\dim
[Y;B]=5.$

$iii)$ $B$ contains all lines $l\in T_{P}(Y),$ i.e. $B=T_{P}(Y).$ In this
case example \ref{e2} shows that $\dim [Sec(X)]=4$.

Let us assume that $\dim (B\cap Y)=0$ and that $B\cap Y$ is supported at two
distinct points $P,Q$ $\in Y,$ at least. By the above analysis we have only
to consider the case in which the intersection is transversal at $P$ and at $%
Q$. In this case the projection of $Y$ from the line $\left\langle
P,Q\right\rangle $ into a generic $\Bbb{P}^{3}$ gives rise to a smooth
quadric, (recall that $Y$ has no trisecant lines), and any linear projection
of a smooth quadric from a point of $\Bbb{P}^{3}$ has $\Bbb{P}^{2}$ as its
image. Hence $\dim [$ $\overline{\pi _{B}(Y\backslash B)}]=2$ and $\dim
[Y,B]=5.$

Now let us consider the case in which $B$ is a plane, but $B\nsubseteq
\left\langle Y\right\rangle .$ Note that $\dim [Sec(X)]=4$ implies that $%
\dim [Y;B]\leq 4$. Hence $T_{P}(Y)\cap B\neq \emptyset $ for any generic
point $P\in Y$ by Lemma \ref{lemT}. Let us consider $B\cap \langle Y\rangle
. $ If $B\cap \langle Y\rangle $ is a point $R,$ we would have: $R\in
T_{P}(Y)$ for any generic $P\in Y$ and this is not possible as $Y$ is not a
cone (recall Remark \ref{remconi}). If $B\cap \langle Y\rangle $ is a line $%
L,$ it is not possible that $L\subseteq T_{P}(Y)$ for any generic $P\in Y$
as $Y$ is not a cone (recall Remark \ref{remconi}). Then we would have: $%
\dim [T_{P}(Y)\cap L]=0$ for any generic $P\in Y$ and for a fixed line $%
L\subset \langle Y\rangle .$ This is not possible: $\langle Y\rangle $ can
be considered as the space of conics lying on some $\Bbb{P}^{2},$ $L$ is a
fixed pencil of conics, $T_{P}(Y)$ is the web of conics reducible as a fixed
line $l_{P}$ and another line. For generic, fixed, $l_{P}$, the web does not
contain any conic of the pencil $L.$

Now let us consider the case in which $B$ is not a plane. As above, $\dim
[Sec(X)]=4$ implies that $\dim ([Y;B])\leq 4.$ Let us consider $M:=\langle
Y\cup B\rangle $ and let us consider the dual varieties $Y^{*}$ and $B^{*}$
in $M^{*}.$ As $Y\neq B$ we get $Y^{*}\neq B^{*}$ (otherwise $Y^{*}=B^{*}$
would imply $Y=B$). Hence the tangent plane $B^{\prime }$ at a generic point
of $B$ is not tangent to $Y.$ By the above arguments we get $\dim
([Y;B^{\prime }])=5.$ It follows that $T_{P}(Y)\cap B^{\prime }=\emptyset $
for the generic point $P\in Y$ by Lemma \ref{lemT}. Therefore $T_{P}(Y)\cap
T_{Q}(B)=\emptyset $ for generic points $P\in Y$ and $Q\in B$ and $\dim
([Y;B])=5$ by Lemma \ref{lemT}.
\end{proof}

\begin{remark}
\label{remconica} A priori, if $\dim [$ $\overline{\pi _{B}(Y\backslash B)}%
]=1$ for a generic $\pi _{B}$ as above, $\overline{\pi _{B}(Y\backslash B)}$
is a smooth conic. In fact $\overline{\pi _{B}(Y\backslash B)}$ is an
integral plane curve $\Gamma .$ Let $f:$ $\Bbb{P}^{1}\rightarrow \Gamma $ be
the normalization map given by a line bundle $\mathcal{O}_{\Bbb{P}^{1}}(e),$ 
$e\geq 1,$ and let $u:Y^{\prime }\rightarrow Y$ be the birational map such
that $\pi _{B}\circ u$ is a morphism; we can assume that $Y^{\prime }$ is
normal. The morphism $u$ induces a morphism $v:Y^{\prime }\rightarrow \Bbb{P}%
^{1},$ set $D:=v^{*}[\mathcal{O}_{\Bbb{P}^{1}}(1)].$ We have $%
h^{0}(Y^{\prime },D)=6$ because $Y$ is linearly normal and the restriction
of $D$ to the fibres of $u$ is trivial. On the other hand, the map $f$
induces an injection from $H^{0}(\Bbb{P}^{1},\mathcal{O}_{\Bbb{P}^{1}}(e))$
into a $3$-codimensional linear subspace of $H^{0}(Y^{\prime },D)$. Hence $%
h^{0}(\Bbb{P}^{1},\mathcal{O}_{\Bbb{P}^{1}}(e))=3$, hence $e=2$ and $\Gamma $
is a conic, necessarily smooth.
\end{remark}

\section{Surfaces having at most two irreducible components}

In this section we study the cases in which $\dim ([A;B])\leq 4,$ where $A$
and $B$ are irreducible surfaces, eventually $A=B.$ The following lemma,
proved by Dale in \cite{d}, is the first step, concerning the case $A=B$.

\begin{lemma}
\label{lemVeronese} Let $A$ be an irreducible surface in $\Bbb{P}^{n}$, then 
$\dim [Sec(A)]\leq 4$ if and only if one of the following cases occurs:

$i)$ $\dim (\left\langle A\right\rangle )\leq 4;$

$ii)$ $A$ is the Veronese surface in $\left\langle A\right\rangle \simeq $ $%
\Bbb{P}^{5};$

$iii)$ $A$ is a cone.
\end{lemma}

\begin{proof}
Firstly let us prove that in all cases $i),$ $ii),$ $iii)$ we have $\dim
[Sec(A)]\leq 4.$ For $i)$ and $ii)$ it is obvious. In case $iii)$ $A$ is a
cone over a curve $C$ and vertex $P$, then $[A;A]$ is a cone over $[C;C]$
and vertex $P$ having dimension $1+\dim ([C;C])$ and $\dim ([C;C])\leq 3.$
Note that, in case $iii),$ $\dim (\left\langle A\right\rangle )$ could be
very big.

Now let us assume that $\dim [Sec(A)]\leq 4$ and that $\dim (\left\langle
A\right\rangle )\geq 5.$ If $Sec(A)$ is a linear space then $\dim
(\left\langle A\right\rangle )\leq 4$. Hence we can assume that $Sec(A)$ is
not a linear space. By \cite{a2}, page. 17, we have $\dim [Sec(A)]-\dim
(A)\geq 2,$ on the other hand $\dim [Sec(A)]-\dim (A)\leq 2$ in any case, so
that $\dim [Sec(A)]-\dim (A)=2.$ By Proposition 2.6 of \cite{a2} we have $%
Vert[Sec(A)]=Vert(A)$. Hence $A$ is a cone if and only if $Sec(A)$ is a cone.

Let us assume that $A$ is not a cone, by the previous argument we know that $%
Sec(A)$ is not a cone. Hence $A$ is an $E_{2,1}$ variety according to
Definition 2.4 of \cite{a2}. Now Lemma \ref{lemVeronese} follows from
Definition 2.7 and Theorem 3.10 of \cite{a2}.
\end{proof}

\begin{lemma}
\label{lemovvio} Let $A,B$ be two distinct, irreducible surfaces in $\Bbb{P}%
^{n},$ $n\geq 3,$ such that $A$ is a cone over an irreducible curve $C$ and
vertex $P.$ Then $\dim ([A;B])=1+\dim ([C;B])$ unless:

$i)$ $\dim (\left\langle A\cup B\right\rangle )\leq 4;$

$ii)$ $B$ is a cone over an irreducible curve $C^{\prime }$ and vertex $P$
or a plane passing through $P.$
\end{lemma}

\begin{proof}
Note that $C$ is not a line as $A$ is not a plane. By Lemma \ref{lemJ}, $5)$
we have $[A;B]$ $=[[P;C];B]=[P;[C;B]]$ which is a cone over $[C;B]$ having
vertex $P.$ If $\dim ([P;[C;B]])=1+\dim ([C;B])$ we are done. If not, we
have $\dim ([P;[C;B]])$ $=\dim ([C;B])$. Hence $[P;[C;B]]$ $=[C;B]$ because $%
[P;[C;B]]\supseteq [B;C]$ and they are irreducible with the same dimension.
In this case we have $P\in Vert([C;B])$ by Proposition 1.3 of \cite{a1}.

If $\dim ([C;B])=2,$ by Lemma \ref{lemcursup} we know that $%
Vert([C;B])=[C;B]=B$ is a plane, but this is a contradiction as $P\in
Vert([C;B])$ and $A$ is not a plane. Assume $\dim ([C;B])=3,$ Lemma \ref
{lemcursup} gives that $Vert([C;B])=$ $[C;B]\simeq \Bbb{P}^{3}$ . Hence $%
A=[P;C]\subset [C;B]\simeq \Bbb{P}^{3}$ and we are in case $i).$

We can assume that $\dim ([C;B])=4$. Hence $\dim ([A;B])=\dim ([P;[C;B]])$ $%
=\dim ([C;B])=4.$ If $\dim (\left\langle A\cup B\right\rangle )=4$ we are in
case $i),$ otherwise $\dim (\left\langle A\cup B\right\rangle )\geq 5.$

Now let us consider generic pairs of points $c\in C$ and $b\in B.$ As $%
[P;[C;B]]$ $=[C;B]$ we have, for generic $(c,b)\in C\times B,$ the union $%
\bigcup\limits_{c\in C,b\in B}(\left\langle P\cup c\cup b\right\rangle )$ is
contained in $[C;B]$ and has dimension $4,$ i.e. $[C;B]=$ $\overline{%
\bigcup\limits_{c\in C,b\in B,generic}(\left\langle P\cup c\cup
b\right\rangle )}.$ If, for generic $(c,b)\in C\times B,$ $\dim
(\left\langle P\cup c\cup b\right\rangle )=1,$ then the lines $\left\langle
P\cup b\right\rangle $ are contained in $A=[P;C]$ for any generic $b\in B,$
it would imply $B\subseteq A:$ contradiction. Hence $\dim (\left\langle
P\cup c\cup b\right\rangle )=2$ for generic $(c,b)\in C\times B.$ As $\dim
([C;B])=4$ to have $\bigcup\limits_{c\in C,b\in B,generic}(\left\langle
P\cup c\cup b\right\rangle )$ of dimension $4,$ necessarily $\left\langle
P\cup c\cup b\right\rangle =\left\langle P\cup c^{\prime }\cup b^{\prime
}\right\rangle $ for infinitely many $(c^{\prime },b^{\prime })\in C\times
B. $ Let us fix a generic pair $(\overline{c},\overline{b}),$ it is not
possible that infinitely many points $c^{\prime }\in C$ belong to $%
\left\langle P\cup \overline{c}\cup \overline{b}\right\rangle ,$ otherwise $%
C $ would be a plane curve and $A$ would be a plane, so there is only a
finite number of points $\overline{c^{\prime }}\in C\cap \left\langle P\cup 
\overline{c}\cup \overline{b}\right\rangle $. Let us choose one of them;
there exist infinitely many points $b^{\prime }\in B$ such that $%
\left\langle P\cup \overline{c}\cup \overline{b}\right\rangle =\left\langle
P\cup \overline{c^{\prime }}\cup b^{\prime }\right\rangle $. Hence there
exists at least one plane curve $B_{\overline{c}}\subset B,$ corresponding
to $\overline{c}$, such that $\left\langle P\cup \overline{c}\cup \overline{b%
}\right\rangle =\left\langle P\cup \overline{c^{\prime }}\cup B_{\overline{c}
}\right\rangle $ $=\left\langle P\cup \overline{c}\cup B_{\overline{c}
}\right\rangle .$ As $\overline{c}$ $\in C$ was a generic point, we can say
that, for any generic point $c\in C,$ there exists a plane curve $%
B_{c}\subset B$ such that, for generic $(c,b)\in C\times B,$ $\left\langle
P\cup c\cup b\right\rangle =\left\langle P\cup c\cup B_{c}\right\rangle .$
If, for generic $c\in C,$ $B_{c}$ is not a line we have $[C;B]=\overline{
\bigcup\limits_{c\in C,b\in B,generic}(\left\langle P\cup c\cup
b\right\rangle )}=\overline{\bigcup\limits_{c\in C,generic}(\left\langle
P\cup c\cup B_{c}\right\rangle )}$ and $\dim \{\overline{\bigcup\limits_{c%
\in C,generic}(\left\langle P\cup c\cup B_{c}\right\rangle )}\}\leq 3,$
because $\left\langle P\cup c\cup B_{c}\right\rangle =\left\langle
B_{c}\right\rangle $ and the set of plane curves $\{B_{c}|c$ generic, $c\in
C\}$ would determine a family of planes of dimension at most $1$. But this
is not possible as $\dim ([B;C])=4,$ then $B_{c}$ must be a line for generic 
$c\in C$ and $B=$ $\overline{\bigcup\limits_{c\in C,generic}(B_{c})}.$

Note that $[C;B]$ must contain $\bigcup\limits_{\overline{c}\in C,fixed,c\in
C,generic}(\left\langle P\cup \overline{c}\cup B_{c}\right\rangle $ for any
generic point $\overline{c}\in C:$ if $[C;B]$ would contain only $%
\bigcup\limits_{c\in C,generic}(\left\langle P\cup c\cup B_{c}\right\rangle $
it would have dimension at most $3.$ Moreover it is not possible that the
lines $\{B_{c}|c$ generic, $c\in C\}$ cut the generic line $\left\langle
P\cup \overline{c}\right\rangle \subset A$ at different points, otherwise $%
A\subset B$. Hence they cut $\left\langle P\cup \overline{c}\right\rangle $
at one point $P(\overline{c})$ and all lines $\{B_{c}|c$ generic, $c\in C\}$
pass through $P(\overline{c}).$ By letting $\overline{c}$ vary in $C$ we get
a contradiction unless $P(\overline{c})=P$ (or $B$ is a plane cutting a
curve on $A,$ but we are assuming $\dim (\left\langle A\cup B\right\rangle
)\geq 5).$ Hence $B$ is covered by lines passing through $P$ and we are in
case $ii).$
\end{proof}

\begin{proposition}
\label{primo passo} Let $V=A\cup B$ be the union of two irreducible surfaces
in $\Bbb{P}^{n}$ such that $\dim [Sec(V)]\leq 4$ and $\dim (\left\langle
V\right\rangle )\geq 5.$ Then:

$i)$ $B$ is the tangent plane at a point $P\in A_{reg}$ and $A$ is a
Veronese surface in $\left\langle A\right\rangle \simeq \Bbb{P}^{5}$ (or
viceversa), in this case $\dim [Sec(A\cup B)]=4;$

$ii)$ $A$ and $B$ are cones having the same vertex;

$iii)$ $A$ is a cone of vertex $P$ and $B$ is a plane passing through $P;$

$iv)$ $A$ is a surface, not a cone, such that $\left\langle A\right\rangle
\simeq \Bbb{P}^{4}$ and such that $A$ is contained in a $3$-dimensional cone
having a line $l$ as vertex, $B$ is a plane such that $B\cap \left\langle
A\right\rangle =l.$
\end{proposition}

\begin{proof}
Obviously if $\dim [Sec(A\cup B)]\leq 4$ we have $\dim [Sec(A)]\leq 4$ and $%
\dim [Sec(B)]\leq 4,$ so that, for both $A$ and $B,$ one of the conditions $%
i),$ $ii),$ $iii)$ of Lemma \ref{lemVeronese} holds.

If $A$ (or $B)$ is a Veronese surface, Proposition \ref{propYB} tells us
that we are in case $i).$ From now on we can assume that neither $A$ nor $B$
is a Veronese surface.

Let us assume that $A$ is a cone of vertex $P,$ over an irreducible curve $%
C. $ If $B$ is a cone of vertex $P$ we are in case $ii).$ Let us assume that 
$B$ is a cone of vertex $P^{\prime }\neq P,$ over an irreducible curve $%
C^{\prime }$, we can assume that $P^{\prime }\notin C$ by changing $C$ if
necessary. By Lemma \ref{lemovvio} and Lemma \ref{lemJ}, $5),$ we have: $%
\dim ([A;B])=1+\dim ([C;[C^{\prime };P^{\prime }]])=1+\dim ([[C;C^{\prime
}];P^{\prime }]=2+\dim ([C;C^{\prime }]\geq 5$ unless $C$ and $C^{\prime }$
are plane curves lying on the same plane (see Lemma \ref{lemcurve}), but in
this case $\dim (\left\langle A=[P;C]\right\rangle )\leq 3,\dim
(\left\langle B=[P^{\prime };C^{\prime }]\right\rangle )\leq 3$ and $\dim
(\left\langle A\cup B\right\rangle )\leq 4.$

Hence we can assume that $B$ is not a cone and therefore $\dim (\left\langle
B\right\rangle )\leq 4$ by Lemma \ref{lemVeronese}. If $B$ is a plane
passing through $P$ we are in case $iii),$ in all other cases we have $\dim
([A;B])=1+\dim ([C;B])\leq 4$ by Lemma \ref{lemovvio}. Hence $\dim
([C;B])\leq 3.$ By Lemma \ref{lemcursup} we know that, in this case, $\dim
(\left\langle C\cup B\right\rangle )\leq 3$ and this is not possible,
otherwise $\dim (\left\langle A\cup B\right\rangle )\leq 4.$

By the above arguments we can assume that $A$ is not a cone. For the same
reason we can also assume that $B$ is not a cone. Hence, by Lemma \ref
{lemVeronese} we have $\dim (\left\langle A\right\rangle )\leq 4$ and $\dim
(\left\langle B\right\rangle )\leq 4$ and $-1$ $\leq \dim (\left\langle
A\right\rangle \cap \left\langle B\right\rangle )\leq 3.$ If neither $A$ nor 
$B$ is a plane, by Lemma \ref{lemJ2}, we have $\dim (\left\langle
A\right\rangle \cap \left\langle B\right\rangle )=3.$ This implies that $%
\dim (\left\langle A\right\rangle )=\dim (\left\langle B\right\rangle )=4,$
otherwise we would have $\left\langle A\right\rangle \subseteq \left\langle
B\right\rangle $ (or $\left\langle A\right\rangle \supseteq \left\langle
B\right\rangle $) and this is not possible as $\dim (\left\langle
A\right\rangle \cup \left\langle B\right\rangle )=\dim (\left\langle A\cup
B\right\rangle )\geq 5.$ Then we can apply Lemma \ref{lemJ3} and we are done.

Hence we can assume that $B,$ for instance, is a plane, $\dim (\left\langle
B\right\rangle )=\dim (B)=2$ and $\dim (\left\langle A\right\rangle )\leq 4.$
If $\dim (\left\langle A\right\rangle )=2,$ $A$ is a plane and it is not
possible that $\dim (\left\langle A\cup B\right\rangle )\geq 5$ and $\dim
([A;B])\leq 4.$ If $\dim (\left\langle A\right\rangle )=3$ we have $%
\left\langle A\right\rangle \cap B$ is a point $R$ as $\dim (\left\langle
A\cup B\right\rangle )=\dim (\left\langle \left\langle A\right\rangle \cup
B\right\rangle )\geq 5,$ then for any point $P\in A_{reg},$ $T_{P}(A)$
passes through $R,$ because $T_{P}(A)$ $\cap B\neq \emptyset $ by Lemma \ref
{lemJ1} $ii)$. Hence $A$ would be a cone with vertex $R$ and this is not
possible$.$ If $\dim (\left\langle A\right\rangle )=4$ we have $\left\langle
A\right\rangle \cap B$ is a line $l,$ as $\dim (\left\langle A\cup
B\right\rangle )=\dim (\left\langle \left\langle A\right\rangle \cup
B\right\rangle )\geq 5,$ and for any generic point $P\in A_{reg},$ $%
T_{P}(A)\cap l\neq \emptyset $ by arguing as above. Let us choose a generic
plane $\Pi \subset \left\langle A\right\rangle $ and let us consider the
rational map $\varphi :A--->\Pi $ given by the projection from $l.$ $\varphi 
$ cannot be constant, because $A$ is not a plane, on the other hand the rank
of the differential of $\varphi $ is at most one by the assumption on $%
T_{P}(A),$ $P\in A_{reg}.$ Hence $\overline{\func{Im}(\varphi )}$ is a plane
curve $\Gamma $ and $A$ is contained in the $3$-dimensional cone generated
by the planes $\left\langle l\cup Q\right\rangle $, where $Q$ is any point
of $\Gamma .$ We get case $iv).$
\end{proof}

\begin{remark}
\label{remfond} Lemma \ref{lemVeronese} and Proposition \ref{primo passo}
give the proof of Theorem \ref{teoint}.
\end{remark}

\section{Surfaces having at least three irreducible components}

In this section we complete the classification of $J$-embeddable surfaces $V$
. By Corollary \ref{corjoin} this is equivalent to assume that $\dim
[Sec(V)]\leq 4$ and by Theorem \ref{teoint} we can assume that $V=V_{1}\cup
...\cup V_{r}$ has at least three irreducible components $V_{i}$ . As any
surface $V$ is $J$-embeddable if $\dim (\left\langle V\right\rangle )\leq 4$
we will always assume that $\dim (\left\langle V\right\rangle )\geq 5.$ Note
that $V$ is $J$-embeddable if and only if $\dim ([V_{i};V_{j}])\leq 4$ for
any $i,j=1,...,r,$ by Corollary \ref{corovvio}.

\begin{lemma}
\label{lem3span5} Let $V=V_{1}\cup ...\cup V_{r}$, $r\geq 3,$ be a reducible
surface in $\Bbb{P}^{n}$ such that $\dim [Sec(V)]\leq 4.$ Assume that there
exists an irreducible component, say $V_{1},$ for which $\dim (\left\langle
V_{1}\cup V_{j}\right\rangle )\geq 5$ for any $j=2,...,r,$ then we have only
one of the following cases:

$i)$ $V_{1}$ is a Veronese surface and the other components are tangent
planes to $V_{1}$ at different points;

$ii)$ $V_{1}$ is a cone, with vertex a point $P,$ and every $V_{j},$ $j\geq
2,$ is a plane passing through $P$ or a cone having vertex at $P;$

$iii)$ $V_{1}$ is a surface, not a cone, such that $\dim (\left\langle
V_{1}\right\rangle )=4$ and $V_{2},...,V_{r}$ are planes as in case $s=1$ of
example \ref{e1bis}.
\end{lemma}

\begin{proof}
Let us consider $V_{1}$ and $V_{2}.$ By assumption $\dim [Sec(V_{1}\cup
V_{2})]\leq 4$ and $\dim (\left\langle V_{1}\cup V_{2}\right\rangle )\geq 5.$
By Proposition \ref{primo passo} we know that one possibility is that $V_{1}$
is a Veronese surface and $V_{2}$ is a tangent plane to $V_{1}.$ In this
case let us look at the pairs $V_{1},V_{j},$ $j\geq 3$; we can argue
analogously and we have $i).$

In the other two possibilities $ii)$ and $iii)$ of Proposition \ref{primo
passo} for $V_{1}$ and $V_{2}$ we can assume that $V_{1}$ is a cone of
vertex $P.$ Now, by looking at the pairs $V_{1},V_{j},$ $j\geq 3$ and by
applying Proposition \ref{primo passo} to any pair, we have $ii).$

In the last case of Proposition \ref{primo passo} we can assume that $V_{1}$
is a surface, not a cone, such that $\dim (\left\langle V_{1}\right\rangle
)=4.$ By looking at the pairs $V_{1},V_{j},$ $j\geq 2$ and by applying
Proposition \ref{primo passo} to any pair, we have any $V_{j}$, $j\geq 2,$
is a plane cutting $\left\langle V_{1}\right\rangle $ along a line $l_{j}$
which is the vertex of some $3$-dimensional cone $E_{j}\subset \left\langle
V_{1}\right\rangle ,$ $E_{j}\supset $ $V_{1}.$ Hence $V$ is a surface as $X$
in case $s=1$ of Example \ref{e1bis}.
\end{proof}

\begin{corollary}
\label{cor3span5} Let $V=V_{1}\cup ...\cup V_{r}$, $r\geq 3,$ be a reducible
surface in $\Bbb{P}^{n}$ such that $\dim [Sec(V)]\leq 4.$ Assume that there
exists an irreducible component, say $V_{1},$ for which $\dim (\left\langle
V_{1}\right\rangle )\geq 5.$ Then we have case $i)$ or case $ii)$ of Lemma 
\ref{lem3span5}.
\end{corollary}

\begin{proof}
As $\dim (\left\langle V_{1}\right\rangle )\geq 5$ we have $\dim
(\left\langle V_{1}\cup V_{j}\right\rangle )\geq 5$ for any $j=2,...,r,$ so
we can apply Lemma \ref{lem3span5}, obviously case $iii)$ cannot occur.
\end{proof}

From now on we can assume that, if $V=V_{1}\cup ...\cup V_{r}$, $r\geq 3,$
is a $J$-embeddable, reducible surface in $\Bbb{P}^{n}$, i.e. $\dim
[Sec(V)]\leq 4,$ then $\dim (\left\langle V_{i}\right\rangle )\leq 4$ for $%
i=1,...,r. $

Let us consider in the following theorems the case in which there exists at
least a component having a $4$-dimensional span.

\begin{theorem}
\label{teocon4nuovo} Let $V=V_{1}\cup ...\cup V_{r}$, $r\geq 3,$ be a
reducible surface in $\Bbb{P}^{n}$ such that $\dim [Sec(V)]\leq 4$ and $\dim
(\left\langle V\right\rangle )\geq 5.$ Assume that $\dim (\left\langle
V_{i}\right\rangle )\leq 4$ for $i=1,...,r$ and that there exists a
component, say $V_{1},$ such that $\dim (\left\langle V_{1}\right\rangle )=4$
and $V_{1}$ is a surface, not a cone, contained in a $3$-dimensional cone $%
E_{2}\subset $ $\left\langle V_{1}\right\rangle $ having a line $l_{2}$ as
vertex. Then:

$i)$ if $E_{2}$ is the unique $3$-dimensional cone having a line as vertex
and containing $V_{1},$ then $V$ is the union of $V_{1},$ planes of $\Bbb{P}%
^{n}$ cutting $\left\langle V_{1}\right\rangle $ along $l_{2},$ cones in $%
\left\langle V_{1}\right\rangle $ whose vertex belongs to $l_{2},$ planes in 
$\left\langle V_{1}\right\rangle $ intersecting $l_{2},$ surfaces in $%
\left\langle V_{1}\right\rangle $ contained in $3$-dimensional cones having $%
l_{2}$ as vertex;

$ii)$ if there exist other cones as $E_{2}$, say $E_{3},...,E_{k},$ with
lines $l_{3},...,l_{k}$ as vertices, then $V$ is the union of $V_{1}$, other
surfaces contained in $E_{2}\cap ...\cap E_{k}$ (if any), planes pairwise
intersecting and cutting $\left\langle V_{1}\right\rangle $ along at least
some line $l_{j},$ cones in $\left\langle V_{1}\right\rangle $ having vertex
belonging to $l_{2}\cap ...\cap l_{k}$ (if not empty), planes in $%
\left\langle V_{1}\right\rangle $ intersecting $l_{2}\cap ...\cap l_{k}$ (if
not empty).
\end{theorem}

\begin{proof}
Note that it is not possible that $\dim (\left\langle V_{1}\cup
V_{j}\right\rangle )\leq 4$ for all $j=2,...,r,$ otherwise $\dim
(\left\langle V\right\rangle )=4$, then there exists at least a component,
say $V_{2},$ such that $\dim (\left\langle V_{1}\cup V_{2}\right\rangle
)\geq 5.$ By applying Proposition \ref{primo passo} to $V_{1}$ and $V_{2}$
we have $V_{2}$ is a plane cutting $\left\langle V_{1}\right\rangle $ along $%
l_{2}.$ Let us consider $V_{j},$ $j\geq 3.$

If $\dim (\left\langle V_{1}\cup V_{j}\right\rangle )\geq 5$ then, by
Proposition \ref{primo passo}, $V_{j}$ is a plane cutting $\left\langle
V_{1}\right\rangle $ along a line $l_{j}$ which is the vertex of some $3$
-dimensional cone $E_{j}\subset $ $\left\langle V_{1}\right\rangle $, $%
E_{j}\supset V_{1}.$

If $\dim (\left\langle V_{1}\cup V_{j}\right\rangle )\leq 4$ then $%
V_{j}\subset \left\langle V_{1}\right\rangle ;$ in this case, to get $\dim
([V_{j};V_{2}])\leq 4$, it must be $T_{P}(V_{j})\cap l_{2}\neq \emptyset $
for any point $P\in (V_{j})_{reg}$ (recall that $V_{2}$ is a plane). Hence,
either $V_{j}$ is a cone whose vertex belong to $l_{2}$, or $V_{j}$ is a
plane intersecting $l_{2}$ or $V_{j}$ is a surface contained in some $3$%
-dimensional cone having $l_{2}$ as vertex.

Now, if $E_{2}$ is the unique cone of its type containing $V_{1},$ then $V$
is as in case $i),$ otherwise we are in case $ii).$
\end{proof}

\begin{theorem}
\label{teocon4} Let $V=V_{1}\cup ...\cup V_{r}$, $r\geq 3,$ be a reducible
surface in $\Bbb{P}^{n}$ such that $\dim [Sec(V)]\leq 4$ and $\dim
(\left\langle V\right\rangle )\geq 5.$ Assume that $\dim (\left\langle
V_{i}\right\rangle )\leq 4$ for $i=1,...,r$ and that there exists a
component, say $V_{1},$ such that $\dim (\left\langle V_{1}\right\rangle )=4$
and $V_{1}$ is not a surface (not a cone) contained in a $3$-dimensional
cone $E\subset $ $\left\langle V_{1}\right\rangle $ having a line as vertex.
Then we have only the following possibilities, all of them obviously
existing:

$i)$ $V$ is the union of cones having as vertex the same point $P$ and
planes passing through $P;$

$ii)$ there exists a flag $P\subset H=K_{1}\cap K_{2}$ where $H\simeq \Bbb{P}%
^{3}$ and $K_{1}\simeq K_{2}\simeq \Bbb{P}^{4};$ $V$ is the union of
surfaces contained in $H$, cones having vertex $P$ and spanning $K_{i},$
cones having vertex $P$ whose $3$-dimensional span cuts a plane on $H,$
planes passing through $P$ and cutting a line on $H$.

$iii)$ there exists a flag $P\subset H\subset K=\left\langle A\right\rangle $
where $H\simeq \Bbb{P}^{3}$ and $K\simeq \Bbb{P}^{4}$ is spanned by a cone $%
A $ having vertex at $P;$ $V$ is the union of surfaces in $H,$ cones in $K$
having vertex at $P,$ planes in $K$ passing through $P$, cones with vertex
at $P$ whose $3$-dimensional span cuts a plane on $H,$ planes passing
through $P$ and cutting a line on $H;$

$iv)$ there exists a flag $P\subset H\subset K=\left\langle A\right\rangle $
where $H\simeq \Bbb{P}^{2}$ and $K\simeq \Bbb{P}^{4}$ is spanned by a cone $%
A $ having vertex at $P;V$ is the union of surfaces in $K$ having a $3$
-dimensional span containing $H,$ planes in $K$ cutting a line on $H,$ cones
in $K$ having vertex at $P,$ planes in $K$ passing through $P,$ cones with
vertex at $P$ whose $3$-dimensional span contain $H$ and planes passing
through $P$ and cutting a line on $H;$

$v)$ there exists a flag $P\subset H\subset K=\left\langle A\right\rangle $
where $H\simeq \Bbb{P}^{1}$ and $K\simeq \Bbb{P}^{4}$ is spanned by a cone $%
A $ having vertex at $P;V$ is the union of surfaces in $K$ having a $3$
-dimensional span containing $H,$ cones in $K$ having vertex on $H,$ planes
in $K$ intersecting $H$ and planes containing $H.$
\end{theorem}

\begin{proof}
Note that it is not possible that $\dim (\left\langle V_{1}\cup
V_{j}\right\rangle )\leq 4$ for all $j=2,...,r,$ otherwise $\dim
(\left\langle V\right\rangle )=4$, then there exists at least a component,
say $V_{2},$ such that $\dim (\left\langle V_{1}\cup V_{2}\right\rangle
)\geq 5.$ By applying Proposition \ref{primo passo} to $V_{1}$ and $V_{2}$
in our assumptions, we have $V_{1}$ is a cone of vertex a point $P$ and $%
V_{2}$ is another cone of vertex $P$ or a plane passing through $P.$

\textbf{Case 1. }Let us assume that there exists another component, say $%
V_{2},$ in $V$ such that $\dim (\left\langle V_{1}\cup V_{2}\right\rangle
)\geq 5$ and $V_{2}$ is a cone of vertex $P.$ Let us put $A=A_{1}:=V_{1}$
and $A^{\prime }:=V_{2}.$ Let us call $A_{j}$ all the components of $V$ such
that $\dim (\left\langle A\cup A_{j}\right\rangle )\geq 5;$ we know that $%
A^{\prime }$ is one of these components$.$ By Proposition \ref{primo passo}
we have any $A_{j}$ is a cone of vertex $P$ or a plane passing through $P.$
We have to consider many possibilities:

$A_{j}^{3}:$ cones such that $\dim (\left\langle A_{j}^{3}\right\rangle \cap
\left\langle A\right\rangle )=3$ and $\dim (\left\langle
A_{j}^{3}\right\rangle )=4$

$A_{j}^{2}:$ cones such that $\dim (\left\langle A_{j}^{2}\right\rangle \cap
\left\langle A\right\rangle )=2$ and $\dim (\left\langle
A_{j}^{2}\right\rangle )=3$ or $4$

$A_{j}^{1}:$ cones such that $\dim (\left\langle A_{j}^{1}\right\rangle \cap
\left\langle A\right\rangle )=1$ and $\dim (\left\langle
A_{j}^{1}\right\rangle )=3$ or $4$

$A_{j}^{0}:$ cones such that $\dim (\left\langle A_{j}^{0}\right\rangle \cap
\left\langle A\right\rangle )=0$ and $\dim (\left\langle
A_{j}^{0}\right\rangle )=3$ or $4,$ in this case the intersection is $P$

$A_{j}^{1p}:$ planes such that $\dim (\left\langle A_{j}^{1p}\right\rangle
\cap \left\langle A\right\rangle )=1$

$A_{j}^{0p}:$ planes such that $\dim (\left\langle A_{j}^{0p}\right\rangle
\cap \left\langle A\right\rangle )=0,$ in this case the intersection is $P.$

Let us call $B_{1},...,B_{t}$ the other components of $V$ such that $\dim
(\left\langle A\cup B_{i}\right\rangle )\leq 4.$

As $\dim (\left\langle A\right\rangle )=4$ we have every $B_{i}\subset
\left\langle A\right\rangle $ and the following possibilities:

$B_{i}^{4}:$ surfaces in $\left\langle A\right\rangle $ such that $%
\left\langle B_{i}^{4}\right\rangle =\left\langle A\right\rangle $

$B_{i}^{3}:$ surfaces in $\left\langle A\right\rangle $ such that $\dim
(\left\langle B_{i}^{3}\right\rangle )=3$

$B_{i}^{2}:$ planes in $\left\langle A\right\rangle $.

A priori, these are the only possibilities for the components of $V.$ Let us
examine when the condition $\dim ([V_{i};V_{j}])\leq 4$ is fulfilled for any 
$i,j$. We know that there exists at least a cone of type $A^{q},$ $%
q=0,1,2,3, $ i.e. $A^{\prime }$ and, of course, the previous condition is
fulfilled for any pair of cones of type $A^{q}$ and planes of type $A^{qp},$ 
$q=0,1$ and for any pair of surfaces $B_{i}$ because all such surfaces are
in $\left\langle A\right\rangle \simeq \Bbb{P}^{4}$. We have only to check $%
\dim ([B_{i};A_{j}]).$

If there exists in $V$ a component $B_{i}$ of type $B^{4}$ then we have $%
\dim (\left\langle B_{i}^{4}\cup A^{\prime }\right\rangle )=\dim
(\left\langle A\cup A^{\prime }\right\rangle )\geq 5,$ so that, by
Proposition \ref{primo passo}, we have $B_{i}^{4}$ is a cone of vertex $P$
and $V$ can contain any number of such cones.

If $V$ contains a cone of type $A^{0},$ or a cone of type $A^{1}$, or a cone
of type $A^{2}$ having a $4$-dimensional span, or a plane of type $A^{0p},$
it is easy to see that every surface $B_{i}^{3}$ must be a cone with vertex
at $P$ because in all these cases $\dim (\left\langle B_{i}^{3}\cup (\text{%
above type of cone and plane)}\right\rangle $ $\geq 5$ and we use
Proposition \ref{primo passo}. Moreover, in all these cases it is easy to
see that every plane $B_{i}^{2}$ (if any) has to pass through $P$. Hence, in
all these cases, we have $i).$ From now on let us assume that $V$ does not
contain cones or planes of the above types.

Now let us distinguish two subcases.

\underline{Firstly}, let us assume that $A^{\prime }$ is of type $A^{3}$.
Hence $\dim (\left\langle A^{\prime }\right\rangle )=4.$

If there exists in $V$ a component $B_{i}^{3}$ let us consider $\dim
(\left\langle B_{i}^{3}\cup A^{\prime }\right\rangle ).$

If $\dim (\left\langle B_{i}^{3}\cup A^{\prime }\right\rangle )\geq 5,$
then, by Proposition \ref{primo passo}, $B_{i}^{3}$ is a cone of vertex $P$
and $V$ can contain any number of such cones. If $\dim (\left\langle
B_{i}^{3}\cup A^{\prime }\right\rangle )\leq 4$, then $\dim (\left\langle
B_{i}^{3}\cup A^{\prime }\right\rangle )=4$ because $\dim (\left\langle
A^{\prime }\right\rangle )=4.$ Hence $\left\langle B_{i}^{3}\right\rangle
=\left\langle A\right\rangle \cap \left\langle A^{\prime }\right\rangle $
for any $i.$

If there exists in $V$ a plane $B_{i}^{2}$ let us consider $\dim
(\left\langle B_{i}^{2}\cup A^{\prime }\right\rangle ).$ If $\dim
(\left\langle B_{i}^{2}\cup A^{\prime }\right\rangle )\geq 5,$ then, by
Proposition \ref{primo passo}, $B_{i}^{2}$ passes through $P$ and $V$ can
contain any number of such planes. If $\dim (\left\langle B_{i}^{2}\cup
A^{\prime }\right\rangle )\leq 4$, then $\dim (\left\langle B_{i}^{2}\cup
A^{\prime }\right\rangle )=4$ because $\dim (\left\langle A^{\prime
}\right\rangle )=4.$ Hence $B_{i}^{2}\subset \left\langle A\right\rangle
\cap \left\langle A^{\prime }\right\rangle $ for any $i.$

The previous remarks prove that in both cases in which in $V$ there exists a
component $B_{i}^{3},$ not a cone with vertex $P,$ or in $V$ there exists a
component $B_{i}^{2},$ not passing through $P,$ the $3$-dimensional linear
space $\left\langle B_{i}^{3}\right\rangle ,$ or $\left\langle B_{i}^{2}\cup
P\right\rangle ,$ is the intersection $\left\langle A\right\rangle \cap
\left\langle A_{j}^{3}\right\rangle $ for any component $A_{j}^{3}$ (recall
that one of them is $A^{\prime }$).

Therefore, in this subcase, $V$ can have many components of type $A^{3}$,
all of these components cut $\left\langle A\right\rangle $ along the same,
fixed, $3$-dimensional space $H=\left\langle A\right\rangle \cap
\left\langle A^{\prime }\right\rangle $ and there can be any number of
surfaces $B_{i}^{3}$ and planes $B_{i}^{2}$ in $H$. $V$ can also contain
cones of type $A^{2},$ but for any surface $B_{i}^{3}$ not a cone with
vertex at $P,$ it must be $\dim (\left\langle B_{i}^{3}\cup
A_{j}^{2}\right\rangle )\leq 4$ and $B_{i}^{3}\subset H$. Hence $%
\left\langle A_{j}^{2}\right\rangle \cap \left\langle A\right\rangle =H.$ $V$
can also contain planes of type $A^{1p},$ but for any surface $B_{i}^{3}$
not a cone with vertex at $P,$ it must be $\dim (\left\langle B_{i}^{3}\cup
A^{1p}\right\rangle )\leq 4$ and $B_{i}^{3}\subset H$. Hence $\left\langle
A^{1p}\right\rangle \cap \left\langle A\right\rangle \subset H.$ Note that,
under these conditions, every plane $B_{i}^{2}\subset H,$ is such that $\dim
(\left\langle B_{i}^{2}\cup A_{j}^{2}\right\rangle )\leq 4$ and $\dim
(\left\langle B_{i}^{2}\cup A^{1p}\right\rangle )\leq 4.$ So we have case $%
ii).$

\underline{Secondly}, let us assume that in $V$ there are not cones of type $%
A^{3}$ and that $A^{\prime }$ is of type $A^{2}.$ Let us consider the linear
span $H$ in $\left\langle A\right\rangle $ given by the planes $\left\langle
A_{j}^{2}\right\rangle \cap \left\langle A\right\rangle $ and lines $%
\left\langle A_{j}^{1p}\right\rangle \cap \left\langle A\right\rangle .$ As $%
A^{\prime }$ is of type $A^{2}$ we have $4\geq \dim (H)\geq 2.$ Note that
for every surface $B_{i}^{3}$ not a cone with vertex at $P$, it must be $%
\left\langle B_{i}^{3}\right\rangle \supset \left\langle
A_{j}^{2}\right\rangle \cap \left\langle A\right\rangle $ and $\left\langle
B_{i}^{3}\right\rangle \supset \left\langle A_{j}^{1p}\right\rangle \cap
\left\langle A\right\rangle ;$ for every plane $B_{i}^{2}$ not passing
through $P,$ it must be $\dim (B_{i}^{2}\cap \left\langle
A_{j}^{2}\right\rangle \cap \left\langle A\right\rangle )\geq 1$ and $\dim
(B_{i}^{2}\cap \left\langle A_{j}^{1p}\right\rangle \cap \left\langle
A\right\rangle )\geq 0.$

If $\dim (H)=4,$ i.e. $H=\left\langle A\right\rangle ,$ $V$ cannot contain
surfaces $B_{i}^{3}$ not cones with vertex at $P$ and $V$ cannot contain
planes $B_{i}^{2}$ not passing through $P.$ So we are in case $i).$

If $\dim (H)=3,$ $V$ can contain surfaces $B_{i}^{3}$ $\subset H,$ not cones
with vertex at $P,$ and planes $B_{i}^{2}\subset H$ not passing through $P.$
So we are in case $iii).$

If $\dim (H)=2,$ $V$ can contain surfaces $B_{i}^{3},$ not cones with vertex
at $P,$ if $\left\langle B_{i}^{3}\right\rangle \supset H$ and $V$ can
contain planes $B_{i}^{2}$ not passing through $P$ if $\dim (B_{i}^{2}\cap
H)\geq 1.$ So we are in case $iv).$

\textbf{Case 2. }Let us assume that $V$ does not contain a component $V_{i}$
, $i\geq 2,$ such that $\dim (\left\langle V_{1}\cup V_{i}\right\rangle
)\geq 5$ and $V_{i}$ is a cone of vertex $P.$ Hence all other components $%
V_{i}$ in $V$ such that $\dim (\left\langle V_{1}\cup V_{i}\right\rangle
)\geq 5$ are planes passing through $P,$ and one of them must exist.

Let us put $A=A_{1}:=V_{1}$ and let $A^{\prime }$ be an other component of $%
V $ which is a plane of vertex $P$ such that $\dim (\left\langle A\cup
A^{\prime }\right\rangle )\geq 5.$ Let us call $A_{j}$ all the planes of $V$
such that $\dim (\left\langle A\cup A_{j}\right\rangle )\geq 5;$ we know
that $A^{\prime }$ is one of these planes$.$ We have to consider these
possibilities:

$A_{j}^{1p}:$ planes such that $\dim (\left\langle A_{j}^{1p}\right\rangle
\cap \left\langle A\right\rangle )=1$

$A_{j}^{0p}:$ planes such that $\dim (\left\langle A_{j}^{0p}\right\rangle
\cap \left\langle A\right\rangle )=0,$ in this case the intersection is $P.$

Let us call $B_{1},...,B_{t}$ the other components of $V$ such that $\dim
(\left\langle A\cup B_{i}\right\rangle )\leq 4.$

As $\dim (\left\langle A\right\rangle )=4$ we have every $B_{i}\in
\left\langle A\right\rangle $ and the following possibilities:

$B_{i}^{4}:$ surfaces in $\left\langle A\right\rangle $ such that $%
\left\langle B_{i}^{4}\right\rangle =\left\langle A\right\rangle $

$B_{i}^{3}:$ surfaces in $\left\langle A\right\rangle $ such that $\dim
(\left\langle B_{i}^{3}\right\rangle )=3$

$B_{i}^{2}:$ planes in $\left\langle A\right\rangle $.

A priori, these are the only possibilities for the components of $V.$ Let us
examine when the condition $\dim ([V_{i};V_{j}])\leq 4$ is fulfilled for any 
$i,j$. We know that there exists at least a plane of type $A^{qp},$ $q=0,1,$
i.e. $A^{\prime }$ and, as in Case 1, we have only to check $\dim
([B_{i};A_{j}]).$

If in $V$ there exists at least a plane $A^{\prime }$ of type $A^{0p}$
intersecting $\left\langle A\right\rangle $ only at $P$ we can argue as in
Case 1 e conclude that surfaces $B_{i}^{4}$ and $B_{i}^{3}$ must be cones of
vertex $P$ and planes $B_{i}^{2}$ passes through $P$. $V$ can contain any
number of such cones and planes and we have case $i)$ again.

Now let us assume that all planes $A_{j}$, $j\geq 2,$ are of type $A^{1p},$
intersecting $\left\langle A\right\rangle $ along a line $l_{j}$ passing
through $P.$ Let $H$ be the linear subspace of $\left\langle A\right\rangle $
generated by the lines $\{l_{j}\},$ we have: $4\geq \dim (H)\geq 1.$

If there exists in $V$ a component $B_{i}^{4}$ then we have $\dim
(\left\langle B_{i}^{4}\cup A_{j}\right\rangle )=\dim (\left\langle A\cup
A_{j}\right\rangle )=5,$ so that, by Proposition \ref{primo passo}, we have $%
B_{i}^{4}$ is a cone whose vertex belong to $l_{j}.$ If $\dim (H)\geq 2$, $%
B_{i}^{4}$ is a cone of vertex $P,$ if $\dim (H)=1,$ i.e. all the lines $%
\{l_{j}\}$ coincide with a line $l,$ then $B_{i}^{4}$ is a cone whose vertex
belongs to $l.$ In both cases $V$ can contain any number of such cones.

If there exists in $V$ a component $B_{i}^{3}$ then we have if $\dim
(\left\langle B_{i}^{3}\cup A_{j}\right\rangle )\geq 5$ then $B_{i}^{3}$ is
a cone of vertex $P$ and $V$ can contain any number of such cones; on the
other hand $\dim (\left\langle B_{i}^{3}\cup A_{j}\right\rangle )\leq 4$ if
and only if $\left\langle B_{i}^{3}\right\rangle \supset l_{j}.$

Hence, if $\dim (H)=4$ it is not possible that $\left\langle
B_{i}^{3}\right\rangle \supset l_{j}$ for any $j$; if $\dim (H)=3$ this is
possible if and only if $\left\langle B_{i}^{3}\right\rangle =H$ for any
component $B_{i}^{3}$; if $\dim (H)\leq 2$ this is possible if and only if $%
\left\langle B_{i}^{3}\right\rangle \supset H$ for any component $B_{i}^{3}$
; in both cases $V$ can contain any number of such surfaces.

If there exists in $V$ a plane $B_{i}^{2}$ it must intersect any line $%
l_{j}. $ If $\dim (H)=4$ this is not possible unless the plane passes
through $P$; if $\dim (H)=3$ this is possible if and only if $%
B_{i}^{2}\subset H$; if $\dim (H)=2$ this is possible if and only if $%
B_{i}^{2}$ intersects $H$ along a line; if $\dim (H)=1$ this is possible if
and only if $B_{i}^{2}$ intersects the only line $l=l_{j}$; in all cases $V$
can contain any number of such planes, possibly passing through $P.$

In conclusion we have cases $i),iii),iv),v),$ respectively, according to the
dimension of $H:4,3,2,1,$ (of course, in the description of $iii)$ and $iv),$
cones of type $A^{2}$ must be removed).
\end{proof}

From now on we can assume that, if $V=V_{1}\cup ...\cup V_{r}$, $r\geq 3,$
is a $J$-embeddable, reducible surface in $\Bbb{P}^{n}$, i.e. $\dim
[Sec(V)]\leq 4,$ then $\dim (\left\langle V_{i}\right\rangle )\leq 3$ for $%
i=1,...,r. $

\begin{theorem}
\label{teolungo} Let $V=V_{1}\cup ...\cup V_{r}$, $r\geq 3,$ be a $J$%
-embeddable, reducible surface in $\Bbb{P}^{n}$, i.e. $\dim [Sec(V)]\leq 4,$ 
$\dim (\left\langle V\right\rangle )\geq 5,$ such that: $\dim (\left\langle
V_{i}\right\rangle )\leq 3$ for $i=1,...,r$ and there exists at least a
pair, $V_{\overline{i}},V_{\overline{j}}$ such that $\dim (\left\langle V_{%
\overline{i}}\right\rangle )=3$ and $\dim (\left\langle V_{\overline{i}%
}\right\rangle \cup \left\langle V_{\overline{j}}\right\rangle )\geq 5.$
Then:

$i)$ $V$ is an union of cones having the same vertex $P$ and possibly planes
passing through $P;$

$ii)$ $V$ is an union of cones having vertex at two different points $P$ and 
$Q,$ such that the linear spans of each pair of cones with different vertex
intersect along a plane and, possibly, surfaces having $3$-dimensional spans
cutting a plane along the linear span of every cone and, possibly, of planes
passing through the line $l:=$ $PQ$ or intersecting each other, intersecting 
$l$ and cutting a line along any other $3$-dimensional $\left\langle
V_{i}\right\rangle ;$

$iii)$ $V$ is an union of cones having vertex at a same point $P$, surfaces
having $3$-dimensional spans cutting a plane along the linear span of every
cone and passing through a fixed line $l\supset P$ and, possibly, of planes
intersecting each other, intersecting $l$ and cutting a line along any other 
$3$-dimensional $\left\langle V_{i}\right\rangle ,$ (if the planes pass
through $P$ the condition holds only for surfaces which are not cones);

$iv)$ $V$ is an union of a plane $A^{\prime }$ passing through a point $P,$
cones having vertex at $P$ whose linear spans contain a fixed plane $\pi
\supset P$, surfaces $V_{k}$ having $3$-dimensional spans containing $\pi $
and cutting a line along $A^{\prime }$ and, possibly, of planes intersecting 
$A^{\prime },$ intersecting each other and cutting a line along any other $3$
-dimensional $\left\langle V_{k}\right\rangle ,$ (if the planes pass through 
$P$ the condition holds only for surfaces which are not cones);

$v)$ $V$ is an union of a plane $A^{\prime }$ passing through a point $P,$
cones having vertex on $A^{\prime }$ and whose linear spans contain a fixed
plane $\pi $, surfaces $V_{k}$ whose linear spans contains $\pi $ and cuts $%
A^{\prime }$ along a line (only if $P\in \pi )$ and, possibly, of planes
intersecting $A^{\prime },$ intersecting each other and cutting a line along
any other $3$-dimensional $\left\langle V_{k}\right\rangle $ (or passing
through all the involved vertices);

$vi)$ $V$ is the union of cones having the same vertex $P$ and the same $3$%
-dimensional linear span $H$, of a plane $A^{\prime }$ intersecting $H$ only
at $P$ and, possibly, of planes passing through $P$ or intersecting $%
A^{\prime }$ and cutting a line on $H$ and intersecting each other.

All described cases can occur.
\end{theorem}

\begin{proof}
Let us put $V_{\overline{i}}=A$ and $V_{\overline{j}}=A^{\prime }.$ By
Proposition \ref{primo passo} we know that we have to consider three cases.
Recall that we have to check that $\dim ([V_{i};V_{j}])\leq 4$ for any $i,j.$

\textbf{Case 1.} $A$ and $A^{\prime }$ are cones with the same vertex $P,$ $%
\dim (\left\langle A\right\rangle \cup \left\langle A^{\prime }\right\rangle
)=6$ and $\left\langle A\right\rangle \cap \left\langle A^{\prime
}\right\rangle =P.$

Let us consider another components $V_{k}$ such that $\dim (\left\langle
V_{k}\right\rangle )=3;$ by Proposition \ref{primo passo} either $V_{k}$ is
a cone with vertex $P$ or $\dim (\left\langle A\right\rangle \cap
\left\langle V_{k}\right\rangle )=\dim (\left\langle A^{\prime
}\right\rangle \cup \left\langle V_{k}\right\rangle )=2,$ but this is not
possible. Let us consider another components $V_{k}$ which is a plane; by
Proposition \ref{primo passo} $V_{k}$ passes through $P$ or $\dim
(\left\langle A\right\rangle \cap \left\langle V_{k}\right\rangle )=\dim
(\left\langle A^{\prime }\right\rangle \cup \left\langle V_{k}\right\rangle
)\geq 1,$ but this is not possible unless $V_{k}$ passes through $P,$ hence
we are in case $i).$

\textbf{Case 2.} There are no pairs as in Case 1, $A$ and $A^{\prime }$ are
cones with the same vertex $P,$ $\dim (\left\langle A\right\rangle \cup
\left\langle A^{\prime }\right\rangle :=M)=5$ and $\left\langle
A\right\rangle \cap \left\langle A^{\prime }\right\rangle $ is a line $l$
passing through $P.$

Let us consider another component $V_{k}$ such that $\dim (\left\langle
V_{k}\right\rangle )=3$ (if any); either $V_{k}$ is a cone of vertex $P$ or $%
l\subset \left\langle V_{k}\right\rangle \subset M$ because $\dim
(\left\langle A\right\rangle \cap \left\langle V_{k}\right\rangle )=\dim
(\left\langle A^{\prime }\right\rangle \cap \left\langle V_{k}\right\rangle
)=2.$ Let us consider the dual space $M^{*}\simeq \Bbb{P}^{5}.$ In $M^{*}$
there are two disjoint lines $a:=\left\langle A\right\rangle ^{*}$ and $%
a^{\prime }:=\left\langle A^{\prime }\right\rangle ^{*}$generating $%
l^{*}\simeq \Bbb{P}^{3}.$ Every surface $V_{k}$ (not a cone with vertex at $%
P $) gives rise to another line $v_{k}=\left\langle V_{k}\right\rangle
^{*}\subset l^{*}$ intersecting both $a$ and $a^{\prime }$ and these lines $%
\{v_{k}\}$ cannot intersect each other out of the skew lines $a$ and $%
a^{\prime }.$ Note that any pair of disjoint lines $v_{k},v_{k^{\prime }}$
(if any) corresponds to a pair of cones in $M$ having the same vertex,
(because $\left\langle V_{k}\right\rangle \cap \left\langle V_{k^{\prime
}}\right\rangle $ is a line, $\dim ([V_{k};V_{k^{\prime }}])\leq 4$ and we
apply Proposition \ref{primo passo}). This vertex is different from $P$ and
we can have at most one vertex of this type, in spite of the number of
disjoint lines $\{v_{k}\}.$ Let us consider the following two possibilities:

$(a)$ there is at least a pair of disjoint lines $\{v_{k}\}$;

$(b)$ all lines $\{v_{k}\}$ pass through a fixed point $R$ of $a$ (or $%
a^{\prime })$.

In case $(a)$ the pair of disjoint lines determines a pair of cones in $M$
having the same vertex $Q\in l.$ The other lines $v_{k}$, if any, correspond
to cones in $M$ having vertex at $Q$ and cutting $\left\langle
A\right\rangle $ and $\left\langle A^{\prime }\right\rangle $ along a plane. 
$V$ can also contain some surface $V_{k}$ such that $\left\langle
V_{k}\right\rangle $ cuts a plane along the linear span of every cone (it is
possible when $\{v_{k}\}$ contains a very few lines) and $V$ can also
contain planes passing through $l$ or intersecting each other, intersecting $%
l$ and cutting a line along any $3$-dimensional linear span $\left\langle
V_{i}\right\rangle $. In general this is not possible, because these planes
give rise to lines in $M^{*}$ intersecting $a,a^{\prime }$ and all the lines 
$v_{k}$, but it is possible when $\{v_{k}\}$ contains a very few lines. We
are in case $ii);$ note that this case is possible, for instance if we take
a smooth quadric in $M^{*}$ and we take the lines of the two different
rulings: each ruling corresponds to a vertex.

In case $(b),$ $\left\langle V_{k}\right\rangle \cap \left\langle A^{\prime
}\right\rangle $ is a fixed plane containing $l,$ while $\left\langle
V_{k}\right\rangle \cap \left\langle A\right\rangle $ is a plane containing $%
l$ and depending on $k.$ As above, $V$ can also contain planes passing
through $l$ or intersecting each other, intersecting $l$ and cutting a line
along any $3$-dimensional linear span $\left\langle V_{i}\right\rangle $,
(if the planes pass through $P$ the condition holds only for surfaces which
are not cones). We are in case $iii);$ note that this case is possible, for
instance when $\{v_{k}\}$ contains only one line.

To conclude Case 2, note that even if $V$ does not contain surfaces $V_{k}$
(not cones with vertex $P)$ the above discussion shows that we are in case $%
iii)$ too.

\textbf{Case 3.} There are no pairs as in Cases 1 and 2, $A$ is a cone with
vertex $P,$ $A^{\prime }$ is a plane passing through $P$, $\dim
(\left\langle A\right\rangle \cup A^{\prime })=5$ and $\left\langle
A\right\rangle \cap A^{\prime }$ $=P.$ In Case 3 any other surface $B,$
component of $V,$ having a $3$-dimensional span $\left\langle B\right\rangle
,$ is such that $\dim (\left\langle A\right\rangle \cap \left\langle
B\right\rangle )\geq 2,$ otherwise we would get Cases 1 or 2 again.

$(a)$ Let us assume that $V$ contains another component $B$ such that $\dim
\left\langle B\right\rangle =3$ and $\left\langle B\right\rangle \neq
\left\langle A\right\rangle .$ Let us call $\pi :=\left\langle
A\right\rangle \cap \left\langle B\right\rangle $ the common plane and let
us call $N:=\left\langle A\cup B\right\rangle \simeq \Bbb{P}^{4}$. Let us
assume that $N\cap A^{\prime }=P.$

In case $(a)$ every component of $V$ contained in $N,$ having a $3$
-dimensional span as $B,$ is a cone of vertex $P$. Hence $P\in \pi ,$
otherwise $\left\langle A\right\rangle =\left\langle B\right\rangle .$ Let $%
V_{k}$ be a component of $V,$ having a $3$-dimensional span, but not
contained in $N.$ We claim that $\left\langle V_{k}\right\rangle \cap N=\pi $
. Hence $\pi \subset \left\langle V_{k}\right\rangle $. In fact, we know
that $\left\langle V_{k}\right\rangle \cap $ $\left\langle A\right\rangle $
and $\left\langle V_{k}\right\rangle \cap $ $\left\langle B\right\rangle $
are planes (if not we would get Cases 1 or 2), hence $\left\langle
V_{k}\right\rangle \cup N\simeq \Bbb{P}^{5}$ and three $3$-dimensional
linear spaces spanning $\Bbb{P}^{5}$, pairwise intersecting along a plane
must contain the same plane, in our case the plane $\pi .$

If $\left\langle V_{k}\right\rangle \cap A^{\prime }=P$ then $V_{k}$ is a
cone of vertex $P,$ if $\left\langle V_{k}\right\rangle \cap A^{\prime }$ is
a line, $V_{k}$ can be any surface having a $3$-dimensional span. $V$ can
contain planes intersecting each other, intersecting $A^{\prime }$, cutting
a line along any $3$-dimensional linear span $\left\langle
V_{i}\right\rangle $, (if the planes pass through $P$ the condition holds
only for surfaces which are not cones). We are in case $i)$ or $iv);$ note
that case $iv)$ is possible: take a pair of planes $A^{\prime }$ and $\pi $
intersecting at a point $P,$ take cones of vertex $P$ whose spans contain $%
\pi $, surfaces whose span contains $\pi $ and a line of $A^{\prime }$
passing through $P,$ planes spanned by lines on $\pi $ and points on $%
A^{\prime }.$

$(b)$ Let us assume that $V$ contains another component $B$ such that $\dim
\left\langle B\right\rangle =3$ and $\left\langle B\right\rangle \neq
\left\langle A\right\rangle .$ Let us call $\pi :=$ $\left\langle
A\right\rangle \cap \left\langle B\right\rangle $ and let us call $%
N:=\left\langle A\cup B\right\rangle \simeq \Bbb{P}^{4}$. Let us assume that 
$N\cap A^{\prime }=l$ is a line and $P\notin \pi $, (it is not possible $%
N\cap A^{\prime }=A^{\prime }$ because $\dim (\left\langle A\right\rangle
\cup A^{\prime })=5$)$.$ Note that $l\cap \pi =\emptyset ,$ otherwise $A\cap
A^{\prime }$ would contain $P$ and $l\cap \pi :$ contradiction. Hence $%
N=\left\langle \pi \cup l\right\rangle $ and $\pi \cap A^{\prime }=\emptyset
.$

By arguing as in case $(a),$ we have every component of $V$ contained in $N,$
and having a $3$-dimensional span as $B,$ is a cone of vertex belonging to $%
l $. Every component $V_{k}$ of $V,$ having a $3$-dimensional span, but not
contained in $N,$ is such that $\pi \subset \left\langle V_{k}\right\rangle
. $ It follows that $\left\langle V_{k}\right\rangle $ $\cap A^{\prime }$
cannot be a line, otherwise this line would cut a point on $\pi $ and $\pi
\cap A^{\prime }$ would be not empty$.$ Therefore $V_{k}$ is a cone having
vertex on $A^{\prime }.$ $V$ can contain planes intersecting $A^{\prime }$,
cutting a line along any $3$-dimensional linear span of the other components
of $V$ (or passing through all the vertices of involved cones: it may
happens, for instance if all the vertices belong to $l$) and intersecting
each other. We are in case $v);$ note that this case can occur: take a pair
of disjoint planes $A^{\prime }$ and $\pi ,$ take cones having vertices on $%
A^{\prime }$ and whose span contain $\pi ,$ take planes spanned by a line on 
$\pi $ and a point on $A^{\prime }.$

$(c)$ Let us assume that $V$ contains another component $B$ such that $\dim
\left\langle B\right\rangle =3$ and $\left\langle B\right\rangle \neq
\left\langle A\right\rangle .$ Let us call $\pi :=\left\langle
A\right\rangle \cap \left\langle B\right\rangle $ and let us call $%
N:=\left\langle A\cup B\right\rangle \simeq \Bbb{P}^{4}$. Let us assume that 
$N\cap A^{\prime }=l$ is a line and $P\in \pi .$ In this case $P\in l$
otherwise $A^{\prime }=\left\langle P\cup l\right\rangle $ and $A^{\prime
}\subset N:$ contradiction.

The only difference with case $(b)$ is that now $V$ can also contain any
surface whose span is $\left\langle \pi \cup l\right\rangle .$ We are in
case $v)$ too.

$(d)$ Let us assume that $V$ does not contain another component $B$ such
that $\dim (\left\langle B\right\rangle )=3$ and $\left\langle
B\right\rangle \neq \left\langle A\right\rangle .$ We are in case $vi),$ and
obviously it can occur.
\end{proof}

From now on we can assume that, if $V=V_{1}\cup ...\cup V_{r}$, $r\geq 3,$
is a $J$-embeddable, reducible surface in $\Bbb{P}^{n}$, i.e. $\dim
[Sec(V)]\leq 4,$ then $\dim (\left\langle V_{i}\right\rangle )\leq 3$ for $%
i=1,...,r$ and $\dim (\left\langle V_{i}\cup V_{j}\right\rangle )\leq 4$ for
any $i,j=1,...,r.$

To complete the classification we prove the following theorem.

\begin{theorem}
\label{teosolo3} Let $V=V_{1}\cup ...\cup V_{r}$, $r\geq 3,$ be a reducible
surface in $\Bbb{P}^{n}$ such that $\dim [Sec(V)]\leq 4$ and $\dim
(\left\langle V\right\rangle )\geq 5.$ Assume that $\dim (\left\langle
V_{i}\right\rangle )\leq 3$ for $i=1,...,r$ and $\dim (\left\langle
V_{i}\cup V_{j}\right\rangle )\leq 4$ for any $i,j=1,...,r.$ Then either $V$
is an union of planes pairwise intersecting at least at a point or the
following conditions hold: $V_{1}\cup ...\cup V_{t}\cup ...\cup V_{r}$ with $%
1\leq t\leq r$ such that

$i)$ $\dim (\left\langle V_{i}\right\rangle )=3$ for any $1\leq i\leq t$ and 
$V_{i}$ is a plane for $t+1\leq i\leq r$ (if any);

$ii)$ $2\leq \dim (\left\langle V_{i}\right\rangle \cap \left\langle
V_{j}\right\rangle )$ for any $i,j=1,...,t;$ $1\leq \dim (\left\langle
V_{i}\right\rangle \cap V_{j})$ for any $i=1,...,t$ and $j=t+1,...,r;$ $%
0\leq \dim (V_{i}\cap V_{j})$ for any $i,j=t+1,...,r.$

Let $V=V_{1}\cup ...\cup V_{r}$, $r\geq 3,$ be a reducible surface in $\Bbb{P%
}^{n}$ such that $\dim (\left\langle V\right\rangle )\geq 5.$ Assume that $%
\dim (\left\langle V_{i}\right\rangle )\leq 3$ for $i=1,...,r$ and that $V$
is either an union of planes, pairwise intersecting at least at a point, or $%
V_{1}\cup ...\cup V_{t}\cup ...\cup V_{r},$ with $1\leq t\leq r,$ satisfying
conditions $i),ii)$ above. Then $\dim [Sec(V)]\leq 4.$
\end{theorem}

\begin{proof}
If $V$ is an union of planes, obviously every pair of planes must intersect
to have $\dim [Sec(V)]\leq 4.$ If not, $V$ is as in $i).$ $ii)$ follows from
the fact that, for any pair $V_{i},V_{j}\in V,$ $\dim (\left\langle
V_{i}\cup V_{j}\right\rangle )=\dim (\left\langle V_{i}\right\rangle \cup
\left\langle V_{j}\right\rangle )\leq 4.$

Conversely: if $V$ is an union of planes intersecting pairwise at least at a
point obviously $\dim ([V_{i};V_{j}])\leq 4$ for any $i,j=1,...,r$. Hence $%
\dim [Sec(V)]\leq 4.$ If $V$ is as in $i)$, condition $ii)$ implies that $%
\dim (\left\langle V_{i}\right\rangle \cup \left\langle V_{j}\right\rangle
)=\dim (\left\langle V_{i}\cup V_{j}\right\rangle )\leq 4$ for any $%
i,j=1,...,r$ . Hence $\dim ([V_{i};V_{j}])\leq 4$ by Lemma \ref{lemJ2}; in
any case $\dim [Sec(V)]\leq 4.$
\end{proof}

\begin{remark}
\label{remes1} Example 1 is a $J$-embeddable surface $V$ considered by
Theorem \ref{teosolo3}.
\end{remark}

\end{document}